\pdfoutput=1
\documentclass{amsart}
\usepackage{latexsym}
\usepackage{amsfonts}
\usepackage{graphicx}
\newcommand{\C}{\mathbb{C}}
\newcommand{\R}{\mathbb{R}}
\newcommand{\Q}{\mathbb{Q}}

\newcommand{\To}{\rightarrow}
\newcommand{\vp}{\varphi}
\newcommand{\comps}{{f_1,\ldots,f_n}}
\newcommand{\vars}{{x_1,\ldots,x_n}}

\theoremstyle{plain}
\newtheorem{Thm}{Theorem}
\numberwithin{Thm}{section}
\newtheorem{Cor}{Corollary}
\newtheorem{Lemma}[Cor]{Lemma}
\newtheorem{Prop}[Cor]{Proposition}

\newtheorem{Conj*}{Conjecture}

\theoremstyle{remark}
\newtheorem{Ex}{Example}

\newcommand{\J}{{J}acobian\ }
\newcommand{\p}{polynomial\ }
\newcommand{\pk}{{P}inchuk\ }
\newcommand{\nz}{nonzero\ }
\newcommand{\xyp}{$(x,y)$-plane\ }
\newcommand{\sa}{semi-algebraic\ }
\newcommand{\fe}{field extension\ }
\newcommand{\Druz}{Dru{\.z}kowski\ }

\begin{document}
\title{{P}inchuk maps, function fields, and
real {J}acobian conjectures}
\author{L. Andrew Campbell}
\address{908 Fire Dance Lane \\
Palm Desert CA 92211 \\ USA}
\email{landrewcampbell@earthlink.net}

\begin{abstract}

Jacobian conjectures (that nonsingular implies invertible) 
for rational everywhere defined maps of $\R^n$ to itself
are considered, with no requirement for a constant Jacobian determinant or a rational inverse. The associated extension of 
rational function fields must be of odd degree and must have
no nontrivial automorphisms. 
The extensions for the Pinchuk counterexamples 
to the strong real Jacobian conjecture
 have no nontrivial automorphisms, but are of degree six. 
The birational case is proved, the Galois case is clarified
but the general case of odd degree remains open. However, certain topological conditions are shown to be sufficient.  Reduction theorems
to specialized forms are proved. 

\end{abstract}

\maketitle

\tableofcontents

\section{Introduction}\label{intro}

The \J Conjecture (JC) asserts that a polynomial map $F: k^n \To k^n$,
where $k$ is a field of characteristic zero, has a polynomial inverse if its \J determinant, $j(F)$, is a \nz element of $k$.
The JC is still not settled for any $n > 1$ and any specific field $k$ of characteristic zero.
It is known, however, that if it is true for $k = \C$ and all $n > 0$, then it is true in every case.
As $j(F)$ is the determinant of the \J matrix of partial 
derivatives of $F$, it is polynomial, and so for $k = \C$ it is a \nz constant if, and only if, it vanishes nowhere on $\C^n$.
That suggested the Strong Real \J Conjecture (SRJC), which
asserts that a polynomial map $F: \R^n \To \R^n$,
has a real analytic  inverse if $j(F)$ vanishes nowhere on $\R^n$. 
However, Sergey {P}inchuk exhibited a family 
of counterexamples for $n=2$, 
now usually called {P}inchuk maps.

Say $F = (\comps)$, with each component a polynomial in $\vars$.
If $j(F)$ is not identically zero, the components of $F$ are algebraically independent over $k$,
so $k[F]=k[\comps]\subseteq k[X]=k[\vars]$ is an inclusion of polynomial algebras.
The extension of function fields $k(F)\subseteq k(X)$ is algebraic,
since both fields have transcendence degree $n$ over $k$,
and is finitely generated, hence
 of finite degree.
In the JC context with $j(F)$ a \nz constant, $F$ has a polynomial inverse,
and hence $k[F]=k[X]$, if
the extension $k(F)\subseteq k(X)$ is Galois,
in particular in the birational case $k(F)=k(X)$.
For $k=\R$ and the SRJC context, that yields polynomial
invertibility in both the birational and {G}alois cases if
$j(F)$ is a \nz constant; but apparently there are no published
invertibility results in either case
if $j(F)$ just vanishes nowhere on $\R^n$.

In section \ref{pkcx} the extension of function fields is investigated for a previously well studied {P}inchuk map. 
A primitive element is found, its minimal polynomial 
is calculated, and the degree ($6$) and automorphism group
(trivial) of the extension are determined. 
That generalizes to  any {P}inchuk map $F$ defined over
any subfield $k$ of $\R$. 
Although $F$ is generically two to one as a polynomial map of
$\R^2$ to $\R^2$, the degree of the associated 
extension of function fields  $k(F) \subset k(X)$ 
is $6$ and $k(X)$ admits no nontrivial automorphism 
that fixes all the elements of $k(F)$ (Theorem \ref{extension}). 
In particular, the extension is not {G}alois

Section \ref{conj} treats the more general case of 
real rational everywhere defined maps $F:\R^n \To \R^n$ with 
nowhere vanishing \J determinant and their associated function 
field extensions. 
If the extension is birational, then $F$ has an inverse of 
the same character as $F$ (Theorem \ref{bicase}). 
If it is {G}alois, then it is birational if $F$ has a real analytic 
inverse (Theorem \ref{galcase}). 
In addition, some knowm special  cases of the SRJC , involving topological conditions, 
are generalized to the rational context. 
Two necessary conditions for invertibility are found to apply to 
the extension: trivial automorphism group and odd degree. 
The degree parity restriction produces modified conjectures, 
in particular a new variant of the SRJC, with hypotheses that 
exclude the {P}inchuk counterexamples. 
Theorems \ref{reduce-1} and \ref{reduce-2} prove 
some reductions to special cases of the sort familiar in the 
ordinary JC context.

Section \ref{sup} is an appendix. It contains additional detailed 
information on the specific \pk map of section \ref{pkcx} 
that is not needed for the proofs there, but can be used to 
verify assertions about the map.

\section{Pinchuk maps}\label{pkcx}

Pinchuk maps are certain polynomial maps $F=(P,Q): \R^2 \to \R^2$
that have an everywhere positive Jacobian determinant $j(P,Q)$,
and are not injective \cite{Pinchuk}. The polynomial $P(x,y)$
is constructed by defining
$t=xy-1,h=t(xt+1),f=(xt+1)^2(t^2+y),P=f+h$.
The polynomial $Q$ varies for different Pinchuk maps,
but always has the form
$Q =q - u(f,h)$, where $q= -t^2 -6th(h+1)$ and
$u$ is an auxiliary polynomial in $f$ and $h$,
chosen so that
$j(P,Q) = t^2 + (t+f(13+15h))^2 + f^2$.

\subsection{A specific Pinchuk map}\label{specific}{\ \newline}

The specific \pk map used in this paper to investigate the
associated extension of function fields is one introduced by Arno van den Essen via an email to colleagues in June 1994. It is defined
\cite{ArnoBook} by choosing
\begin{equation}\label{ueq}
u =
170fh + 91h^2 + 195fh^2 + 69h^3 + 75fh^3+
\frac{75}{4}h^4.
\end{equation}
The total degree in $x$ and $y$ of $P$ is $10$ and that of $Q$ is $25$. 
%%%%%%%%%%%%%%%%%%%%%%%%%%%%%%%
%%%%%%%%%%%%%%%%%%%%%%%%%%%%%%%
The image, multiplicity  and asymptotic behavior of $F$
were studied in
\cite{PPR,Picturing,PPRErr,aspc}. Its asymptotic variety, $A(F)$,
is the set of points in the image plane that are finite limits of the value of $F$ along curves that tend to infinity  in the \xyp \cite{asymptotics,asympvals}. It may alternatively be defined as the set of points in the image plane that have no neighborhood with
a compact inverse image under $F$
\cite{notproper,realtrans,geometry}.
It is a topologically closed curve in the
image $(P,Q)$-plane and is the image of a real line under a bijective \p parametrization; Its Zariski closure has one additional point not on the curve, so it is a \sa variety, but not an actual real algebraic variety.
It is depicted below using
differently scaled $P$ and $Q$ axes. It intersects the vertical axis at $(0,0)$ and $(0,208)$. Its leftmost point is $(-1,-163/4)$,
and that is the only singular point of the curve.

\begin{figure}[ht]
%\vspace{-0.5in}
\centerline{
\includegraphics[height=2in,width=6in]{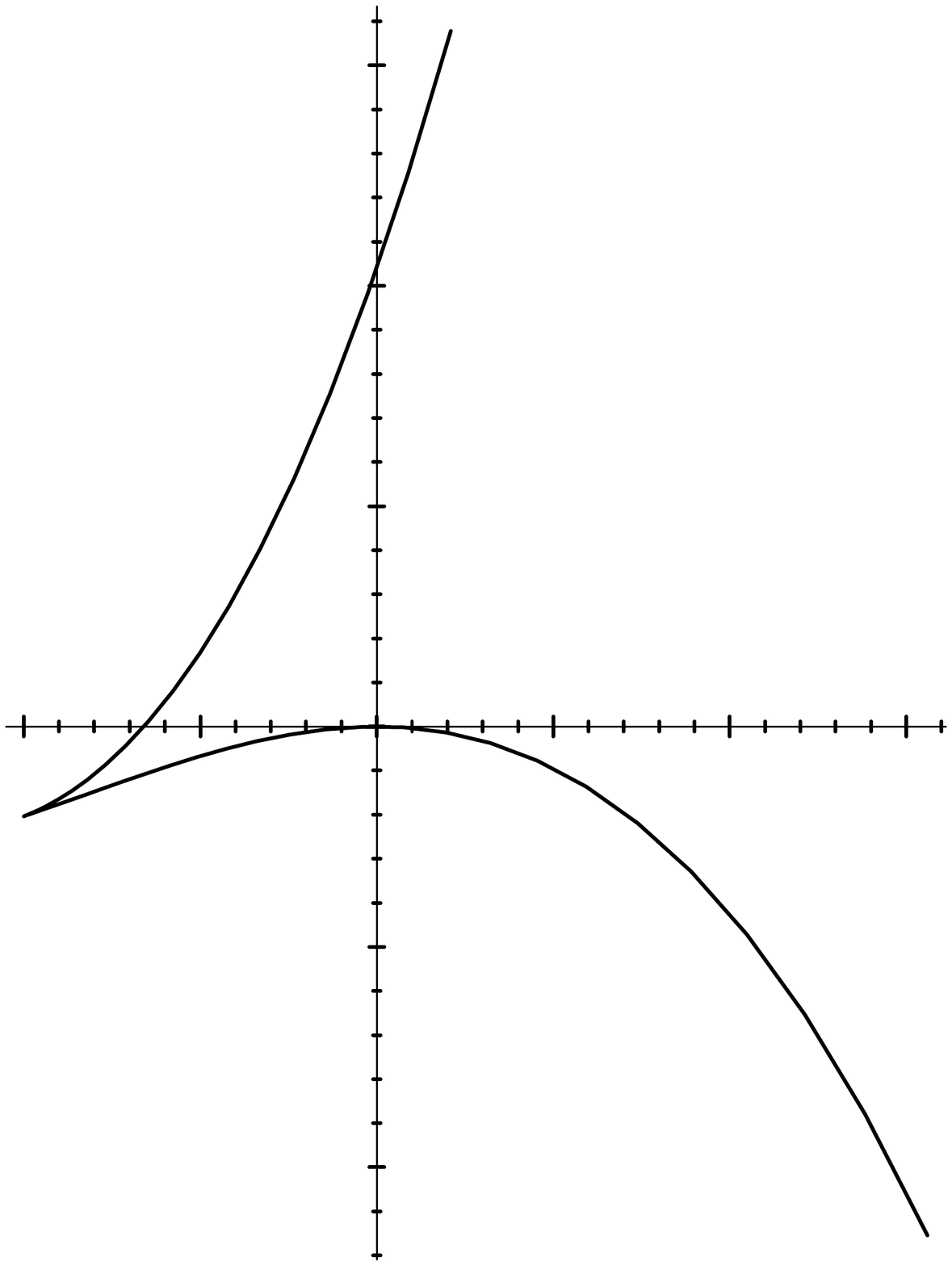}}
\caption{The asymptotic variety of the Pinchuk map $F$.}
\end{figure}

The points $(-1,-163/4)$ and $(0,0)$ of $A(F)$ have no inverse image under $F$, all other points of $A(F)$ have one inverse
image, and all points of the image plane not on $A(F)$ have two.

The inverse image of $A(F)$ under $F$ is the disjoint union
of three smooth curves, each of which is a topological
line that extends to infinity at both of its ends.
 The curves partition their complement in the \xyp
into four simply connected unbounded open sets.
Those regions are mapped homeomorphically to their images, two each to the regions on either side of $A(F)$.

\subsection{Minimal polynomial calculation}\label{ext}{\ \newline}

This paragraph is a summary of some key facts from previously cited work on $F$.
 A general level set $P=c$
in the $(x,y)$-plane has a rational parametrization.
Specifically, for any real $c$ that is not $-1$ or $0$,
the equations 
$$x(h) = \frac{ (c-h)(h+1) }{  (c-2h-h^2)^2 }$$
$$y(h) = \frac{ (c-2h-h^2)^2(c-h-h^2) }{ (c-h)^2 },$$
define a rational map pointwise on a real line with coordinate $h$, except where a pole occurs.
The use of $h$ as a parameter and the equality $P=c$ 
are consistent: on
 substitution into the
defining equations $t=xy-1,h=t(xt+1),f=(xt+1)^2(t^2+y),P=f+h$,
the expression  $h(x(h),y(h))$ simplifies to  $h$,
and $P(x(h),y(h))$ to $c$.
There is always a pole at $h=c$ and
$Q(x(h),y(h))$ tends to $-\infty$ as the pole
is approached from either side. 
Also, $Q(x(h),y(h))$ tends to $+\infty$ as $h$ tends 
to $+\infty$ and 
as $h$ tends to $-\infty$. 
If $c > -1$, there are two additional poles
at $h = -1 \pm \sqrt{1+c}$ and
$Q(x(h),y(h))$ tends to a finite asymptotic value at
each of these poles  as the pole
is approached from either side.
In that case
the asymptotic values are distinct and are the values of $Q$ at the two points of intersection of the vertical line $P=c$ and $A(F)$ in the $(P,Q)$-plane. 
The level sets $P=c$ are disjoint unions of 
their connected components, which are 
curves that are smooth (because of the \J condition) and tend 
to $\infty$ in the $(x,y)$-plane at both ends. The number of 
curves is two if $c < -1$, and four if $-1 < c \ne 0$. Even 
the two exceptional values fit this pattern, although they require 
different rational parametrizations, with $P=-1$ consisting of 
four curves and $P=0$ of five.

As a concrete illustration, consider the case $P=3$.
Detailed justifications are omitted.
The points of intersection of the vertical line $P=3$ and $A(F)$
are $a = (3,14965/4)$ and $b =(3, -4235/4)$.
The poles are at $h=-3$, $h=1$, and $h=3$.
As $h$ varies from $-\infty$ to $3$ 
the image point $F(x(h),y(h))$  moves down
the vertical line $P=3$
from infinity to $a$, skips $a$ because of the first pole,
continues down to $b$, skips $b$ at  the second pole, then
traces out the rest of the line to negative infinity
as the third pole is approached from below.
On the other side of the third pole the
entire line is retraced from
 negative infinity  to positive infinity without any skips.
That makes it obvious why $a$ and $b$ each have
exactly one inverse image in the $(x,y)$-plane.

$F$ is not birational, because it is generically two to one.
Throughout the remainder of this section, let $k=\R$.
To begin the exploration of the \fe $k(P,Q)\subset k(x,y)$,
rewrite the parametrization above in terms of $f$ and $h$,
using the relations  $P=c$ and $P=f+h$ to obtain
$$x =  f(h+1)(f-h-h^2)^{-2}$$
$$y = (f-h-h^2)^2(f-h^2)f^{-2},$$
which are identities in $k(x,y)$
(and so  $k(x,y) =  k(f,h)$).
It follows that
$xy = (h+1)(f-h^2)/f,
t  = xy -1 = [(h+1)(f-h^2)-f]/f
    = [fh-h^2-h^3]/f 
   = (h/f)[f-h(h+1)],
 q = -t^2-6th(h+1)
   = -h^2f^{-2}\{[f-h(h+1)]^2  + 6(h+1)[f-h(h+1)]f\}.$
In fact,
\begin{align*}
q &= -h^4(h+1)^2/f^2 + [2h^3(h+1)+6h^3(h+1)^2]/f
     + [-h^2 -6h^2(h+1)] \\
  &= -h^4(h+1)^2/f^2
   + h^3(h+1)(6h+8)/f
   -h^2(6h+7).
\end{align*}
Using that equation,  the definition $Q=q-u(f,h)$, and 
equation \ref{ueq} for $u$ one can express $Q$ in terms 
of $f$ and $h$ alone. 
Clearing denominators $f^2Q = f^2q - f^2u$, or,
arranged by powers of $f$, 
\begin{align*}
f^2Q &=  -h^4(h+1)^2 \\
      &+f [h^3(h+1)(6h+8)] \\
      &+f^2[-h^2(6h+7)-91h^2 -69h^3-(75/4)h^4] \\
      &+ f^3[-170h-195h^2-75h^3].
\end{align*}
Now substitute $P-h$ for $f$ and collect in powers of $h$
to obtain a \p relation
\begin{equation}\label{Rmin}
(197/4) h^6 + \cdots +(2PQ - 170P^3)h -P^2Q = 0.
\end{equation}
Let $R(T)$  be the corresponding polynomial in $T$ with root $h$.  
$R(T)$ is explicitly written out in full 
in the Appendix (section \ref{sup}). 
It is clear, even without
an explicit formula,  that the coefficient of each power of $T$ is a \p in $P$ and $Q$ with rational  coefficients,
and has  total degree in $P$ and $Q$ at most $3$.
Since the leading coefficient of $R(T)$ is a real constant,
the fact that $R(h) = 0$  shows that $h$ is
integral over $k[P,Q]$.

Let $m(T)$ be the polynomial in $k[P,Q][T]$, $ T$ an indeterminate,
which has leading coefficient $1$ and satisfies $ m(h) = 0$  in $k[x,y]$,
and which is of minimal degree. 
 Clearly $m$ is  irreducible in
$k[P,Q][T] = k[P,Q,T]$ and hence by the Gauss Lemma, in
$k(P,Q)[T]$. That implies that $m$ is also of minimal degree
over $k(P,Q)$, that $m$ divides any polynomial in $k[P,Q][T]$
with $h$ as a root, and that $m$ is unique.

Note the following $k$-linear field inclusions
$$k(P,Q)\subset k(P,Q)(h) = k(f,h) = k(x,y).$$
Next consider the k-algebra homomorphisms
$$k[P,Q]\subset k[P,Q][h] \subseteq k[x,y]$$
and the corresponding regular maps of affine real
algebraic varieties
$$k^2 \To\rm{Zeroset}(m) \To k^2$$
with the first map sending $(x,y)$ to $(P(x,y),Q(x,y),h(x,y))$
and the second the projection onto the first two components. 
The first map is birational ($k(P,Q)(h) = k(x,y)$) and the second
is finite (topologically proper with an overall bound on the
number of inverse images of points in the codomain) by
integrality. Incidentally, that shows that the inclusion
$k[P,Q][h] \subseteq k[x,y]$ is actually strict, since $F$
is not topologically proper.

If we fix any point $w$ in the$(P,Q)$-plane, we may consider $m$
as a polynomial in $T$ with real coefficients and real roots that
determine the points projecting onto $w$ under the second map.
So we call the algebraic surface $m=0$ the variety of real roots
of $m$ and if $m(w,r) =0$ we say that $r$ is a root of $m$
over $w$.
Note that for any point of the $(x,y)$-plane, $h(x,y)$ is a real
root of $m$ over $w=F(x,y)$.

\begin{Lemma}\label{gl}
For generic $w$, $m$ has exactly two real roots over $w$, they are simple and distinct, $w$ has exactly two inverse images $v$ and $v'$ under $F$, and the real roots of $m$ over $w$ are $r=h(v)$ and $r'=h(v')$.
\end{Lemma}

\begin{proof}
Take a bi-regular isomorphism from a Zariski open subset $O$
of the $(x,y)$-plane to a Zariski open subset of the variety of real
roots of $m$. The image is a nonsingular surface $S$.
In the usual (strong) topology it has
a finite number of connected components that are open subsets of $S$ and of the variety of real roots.
Take the union of i) the image of the complement of $O$ under $F$, ii) the projection of the complement of $S$, and iii)   $A(F)$, the asymptotic variety of $F$.
From the Tarski-Seidenberg projection property  and other basic tools of real \sa geometry, the union is \sa  of maximum dimension $1$. Take $w$ in the complement of the Zariski closure of that union. Any root that lies over $w$ is a nonsingular point of the variety of real roots (by construction), and so is a simple, not multiple, real root.
There are exactly two points, say $v$ and $v'$, that map to $w$ under $F$. Their images under $h$, $r$ and $r'$, lie over $w$.
By construction $(w,r)$ and $(w,r')$ lie in $S$, and $v$ and $v'$ lie in $O$. Hence $r$ and $r'$ are distinct. No point in the complement of $O$ can map to $(w,r)$ or $(w,r')$
(by construction), and $v$ and $v'$ are the only points
in $O$ that do so.
\end{proof}

\begin{Cor}
The $T$-degree of $m$ is even.
\end{Cor}

\begin{proof}
The complex roots over $w$ that are not real occur in
complex conjugate pairs.
\end{proof}

Let $m_0$ be the term of $m(T)$ of degree $0$ in $T$. Clearly, 
$m_0 \in k[P,Q]$is not the zero polynomial. Since $m(T)$  
divides $R(T)$, $m_0$ divides $P^2Q$. Since the $T$-degree 
of m is even, $m_0(w)$ is the product of all the roots, real and complex, of $m$ over $w$, for any $w$ in the $(P,Q)$-plane. 

\begin{Prop}
$m_0$ is a positive constant multiple of$ -P^2Q$.
\end{Prop}

\begin{proof}
$F(1,0)=(0,-1)$ and $h(1,0)=0$, so $m_0(0,-1)=0$. 
This  shows that $P$ must divide $m_0$, for otherwise
$m_0(0,-1)$ would be nonzero. Next, $F(1,1)=(1,0)$ and $h(1,1)=0$, so $m_0(1,0)=0$. So $Q$ divides $m_0$.
Next, consider the union of
the vertical lines $P=c$ in the $(P,Q)$-plane, for $2 < c< 4$.
At least one such line must contain a point $w$ that is generic
in the sense of Lemma \ref{gl}, for otherwise there would be an
open set of nongeneric points. Choose such a $c$, and note that
all but finitely many points of the line $P=c$ are generic.
The level set $P=c$ in the $(x,y)$-plane has the rational parametrization by $h$ already described, with a pole at $h=c$, at which $Q$ tends to $-\infty$. 
Take a point $w = (c,d)$ with $d$ negative and sufficiently large.
Then $w$ will be generic and its two inverse images under $F$ will
have values of $h$ that approach $c$ as $d$ tends to $-\infty$,
one value of $h$ less than $c$ and the other greater,
The product of all the
roots of $m$ over $w$ will be positive, since the nonreal roots
occur as conjugate pairs. Since $P$ is positive and $Q$ negative,
the numerical coefficient of $m_0$ must be negative, regardless of
whether $m_0$ is exactly divisible by $P$ or by $P^2$.
Finally, make a
similar argument for a suitable line $P=c$, with $-4 < c < -2$.
Consideration of signs shows that $m_0$ must be divisible by $P^2$,
which yields the desired conclusion.
\end{proof}

\begin{Cor}\label{not2}
The $T$-degree of $m$ is not $2$.
\end{Cor}

\begin{proof}
If $m$ has degree $2$ and $w$ is generic, then $m_0(w)$ is exactly the product of the two real roots of $m$ over $w$.  But for the last two examples considered in the previous proof, the product tends to $c^2$ as $h$ tends to $c$, whereas $m_0(w)$ is unbounded.
\end{proof}

\begin{Prop}\label{R=cm}
$R(T) = (197/4)m(T)$.
\end{Prop}

\begin{proof}
As $m(T)$ is a nonconstant divisor of $R(T)$ of even $T$-degree not
equal to $2$, it remains only to show that the degree of
$m$ in $T$ is not $4$.
Suppose to the contrary that
$(m_4T^4+m_3T^3+m_2T^2+m_1T+m_0)
(d_2T^2+d_1T+d_0) = R(T)$,
where the first factor is $m(T)$, the second is a polynomial $D(T)$
 of degree $2$ in $T$,  and all the coefficients shown are in
$k[P,Q]$.
Equating leading and constant terms on both sides, one
finds that $m_4=1$, $d_2=197/4$ and $d_0$ is a positive constant.
The coefficient $d_1$ must also be constant. For if not,
$j = \deg^t(d_1) > 0$, where $\deg^t$ temporarily denotes the total degree in $P$ and $Q$.
As noted earlier, that $\deg^t$ is at most $3$
for every coefficient of $R(T)$.
Starting with $\deg^t(m_0)=3$ and equating in turn terms of
$T$-degree $1$ through $4$ on both sides of the equation
assumed for $R(T)$, one
readily finds that $\deg^t(m_4)=3+4j$. But $m_4=1$, a contradiction.
Thus $D(T)$ has constant coefficients. Next, set $P=0$ in $R(T)$,
obtaining $(197/4)T^6 + 104T^5 + 63T^4 - QT^2$. That result
can be found easily by setting $f=-h$ in the rational equation
for $Q$ in terms of $f$ and $h$. Further setting $Q=0$, one finds that the resulting \p in $T$ alone 
factors as $T^4((197/4)T^2 + 104T+ 63)$. Clearly $D$ must
be exactly the quadratic factor shown. But if $D(T)$  divides $R(T)$,
setting $P=0$ implies it must also divide $-QT^2$, which is absurd.
That contradiction shows that the original assumption to the contrary, that $m$ has $T$-degree $4$, is false.
\end{proof}

\begin{Cor}\label{six}
The \fe $k(P,Q) \subset k(x,y)$ is of degree $6$.
\end{Cor}

\begin{proof}
Clear.
\end{proof}

\subsection{Automorphisms of the extension}\label{aut}{\ \newline}

This section examines automorphisms of the extension 
$k(P,Q)\subset k(x,y)$; that is, field automorphisms of 
$k(x,y)$ that fix every element of $k(P,Q)$. Again, assume 
throughout that $k=\R$. 

First, consider the geometry over $\R$. 
Let $Z=F^{-1}(A(F))$. 
For any $(x,y)\notin Z$ there is a unique different point 
$(x',y')\notin Z$ with the same image under $F$. 
 There is an obvious Klein four group 
$\{e,\sigma,\sigma',\tau\}$ 
of $F$-preserving 
transformations of the complement of $Z$, where 
$e$ is the identity map, 
$\sigma$ interchanges inverse images over points that lie 
on one side of $A(F)$, 
$\sigma'$ is defined similarly for the other side of $A(F)$, 
and $\tau=\sigma\sigma'=\sigma'\sigma$ 
leaves no point invariant, always interchanging $(x,y)$ 
and $(x',y')$. These geometric involutions are Nash 
diffeomorphisms, 
that is, they are real analytic and semi-algebraic. Except for the identity map 
$e$, these transformations cannot be even locally extended 
analytically to any point $z \in Z$. For otherwise, $z$ would be 
a fixed point, and since $F$ is a local diffeomorphism the map 
would be the identity map on a neighborhood of $z$, thus over 
both sides of $A(F)$.

Next, the algebra. Suppose $\varphi$ is an automorphism of 
$k(x,y)$ that is not the identity but 
fixes every element of $k(P,Q)$. 
If $\vp$ preserves $h$, then it also preserves $x$ and $y$, since 
they are rational functions of $P$ and $h$, namely 
\begin{equation}\label{ratexp}
x= \frac{ (P-h)(h+1) }{  (P-2h-h^2)^2 }
\text{  and  }
y= \frac{ (P-2h-h^2)^2(P-h-h^2) }{ (P-h)^2 }.  
\end{equation}
So $\vp(h)=h'\ne h$ and, furthermore, by the identity
principle for rational functions over $\R$, they cannot be equal on any nonempty open subset of the $(x,y)$-plane on which $h'$ is defined. 
Since
$\vp$ preserves both components of $F$, the fact that its geometric realization cannot be the identity even locally means that it must be $\tau$ (see above) wherever both are defined. 
That implies that there can be at most one such a nonidentity 
automorphism $\vp$. If it exists, then
the rational function $h'$ is analytic at any point 
$(x,y) \notin Z$, since $h'(x,y)=h(x',y')=h(\tau(x,y))$.

That reduces the question of the existence of $\vp$ 
to the following one. 
Is $h'$, a well defined 	real analytic function on the complement of $Z$ in the $(x,y)$-plane, in fact a real rational function?

\begin{Lemma}\label{existence}
There are three component curves of the level set $P=0$ on 
which $h$ is nonconstant and vanishes nowhere. 
On those curves 
$Q=Q(h)=h^2((197/4)h^2+104h+63)$ and is
everywhere positive. 
There is one point of $Z=F^{-1}(A(F))$ on the three curves. 
At all of their other points, $h'$ satisfies both 
$h' \ne h$ and $Q(h')=Q(h)$. 
\end{Lemma}

\begin{proof}
Set $P=0$ in equation \ref{ratexp}. The resulting
 rational functions $x(h),y(h)$ are defined everywhere
except at $h=-2$ and $h=0$. 
That yields three curves parametrized by $h$.
Since $h(x(P,h),y(P,h))$ simplifies to $h$ for the rational 
functions in equation \ref{ratexp}, it follows that 
$h(x(h),y(h))=h$, and hence $h$ assumes every real value 
exactly once on these curves, except that $-2$ and 
$0$ are never assumed. 
Since every level set $P=c$ is a finite disjoint union of 
closed connected smooth curves unbounded at both ends, each of 
the three curves is a connected component of $P=0$.

Next, set $P=0$ in equation \ref{Rmin} of section \ref{ext},
 which is 
the relation $R(h)=0$, where $R(T)$ is 
that section's minimal degree, 
but not monic, polynomial with root $h$. 
The result, which in essence already appeared in the proof 
of Proposition \ref{R=cm}, is 
$(197/4)h^6+104h^5+63h^4-Qh^2=0$. 
On the curves, $h \ne 0$, so the claimed formula for 
$Q$ follows. 
Furthermore, $Q$ is positive there, 
since $(197/4)h^2+104h+63$ has negative 
discriminant.

So $F$ maps points on the three curves to the positive $Q$-axis. 
Routine calculation of the derivative of $Q$ shows that $Q$ is 
monotonic decreasing for $h<0$ and monotonic 
increasing for $h>0$. 
Considering the graph of $Q$, one concludes that every positive 
real is the value of $Q$ exactly twice, for \nz values of 
$h$ of opposite signs. Those values of $h$ all correspond to 
unique points of the curves, except $h=-2$. As $Q(-2)=208$, 
the point $(0,208)$, which is the only point of $A(F)$ on 
the positive $Q$-axis, has as its unique inverse under $F$ the 
point $(x(h'),y(h'))$, where $h'$ is the positive real 
satisfying $Q(h')=208$. 
\end{proof}

Remark. To clarify, 
there are two additional component curves of the level set 
$P=0$. On them $h=0$ identically. They have a rational 
parametrization by $t$ with a pole at $t=0$ and $Q=-t^2$ is
everywhere negative on them. They contain no point of $Z$, 
$t'=-t$, and $h'=0$.

\begin{Lemma}\label{irrational}
Let $h'$ be any real rational function of $h$ that 
satisfies $Q(h')=Q(h)$ for infinitely many values 
of $h \in \R$. Then $h'=h$. 
\end{Lemma}

\begin{proof}
Suppose $h'=a/b$ for polynomials 
$a,b \in \R[h]$, of respective degrees $r,s$, with no common 
divisor. From $Q(h')=Q(h)$ one obtains 
\begin{equation}\label{abeq}
a^2((197/4)a^2+104ab+63b^2)=b^4h^2((197/4)h^2+104h+63),
\end{equation}
a polynomial equality that is true for all real $h$.  
Since $a/b$ tends to $\infty$ as $h$ does, 
$r>s$. Counting degrees $4r=4s+4$, so $r=s+1$. 
The factor in parentheses on the left is quadratic and 
homogeneous in 
$a$ and $b$ and has negative discriminant, so it is zero for real 
$a$ and $b$ only if both are zero. But that cannot occur for 
any real $h$, for then $a$ and $b$ would have a common root, 
hence a common divisor. 
Thus that factor has only complex roots. It follows that 
$h^2$ divides $a^2$, and so $h$ divides $a$. 
That means that $a$ and $bh$ share a root, each having 
$h$ as a factor. As the quadratic factor in parentheses on the 
right is not zero for any real $h$, any further real roots 
(at $h=0$ or not) in the two sides of equation \ref{abeq}
would be shared by $a$ and $b$. Again, that is not possible, 
and therefore $a$ has no further real roots and all the roots of 
$b$ are complex. In particular $s$, the degree of $b$, must 
be even. No complex root of $b$ can be a root of $a$, as that 
would imply a common real irreducible quadratic factor. 
So it must be a root of the parenthetical factor 
on the left. Counting roots with multiplicities, 
$2r=2s+2 \ge 4s$, so $s \le 1$. Since $s$ is even, it 
must be $0$, and so $h'=\lambda h$ for a 
\nz $\lambda \in \R$. 
Then for any fixed $h \ne 0$, $Q(\lambda^ih)=Q(h)$ is 
independent of $i>0$, and so $\lambda$ has absolute value 
$1$. It cannot be $-1$, because $Q(h)-Q(-h)=208h^3$. 
\end{proof}

\begin{Prop}\label{notg}
The group of automorphisms 
of the  \fe $k(P,Q) \subset k(x,y)$ is trivial. 
In particular, the extension is not {G}alois. 
\end{Prop}

\begin{proof}
If the group contains a nontrivial automorphism $\vp$, 
then $h'=\vp(h)=h(x',y')$ (see above) belongs to
$\R(x,y)=\R(P,h)$. As $\R(P,h)$ is a rational function field 
in two algebraically independent elements over $\R$, the 
restriction of $h'$ to the level set $P=0$ must either be 
generically undefined (uncanceled $P$ in the denominator) or 
a rational function of $h$. 
The first case is ruled out by Lemma \ref{existence}, which 
also contradicts Lemma \ref{irrational} in the second case.
\end{proof}

\subsection{All {P}inchuk maps}\label{main}{\ \newline}

From a geometric point of view,
any two different \pk maps are very closely related.
More specifically, if $F_1=(P,Q_1)$ and $F_2=(P,Q_2)$
are \pk maps then they have the same first component, $P$,
and their second components satisfy $Q_2=Q_1+S(P)$
for a \p $S$ in one variable with real coefficients.
As maps of $\R^2$ to $\R^2$, therefore, they differ only by
a triangular \p automorphism of the image plane. In effect,
all \pk maps share the same geometry. Indeed, the set
$Z=F	^{-1}(A(F))$ and the real analytic function 
$h'$ of the previous section \ref{aut} are exactly the same
for all \pk maps. 
Moreover,  all \pk maps are generically two to one,
 and their asymptotic varieties
have algebraically isomorphic embeddings in the 
image plane.

Remark. 
In \cite{GeoPinMap} Janusz Gwo{\'z}dziewicz studied
a Pinchuk map of total degree $40$
and noted the single point in the Zariski closure of
the asymptotic variety not on the variety itself.
He first brought to my attention
the \p relation between different \pk maps
in an informal communication in 2009..
An algebraically oriented proof, along lines suggested by
Arno van den Essen, can be found in \cite{aspc}.

Let $F$ be the same \pk map as before.  It is defined over
$\Q$. In fact, not only do $P$ and $Q$ have rational coefficients,
but so do $h$ and all terms of the minimal \p $m$ for $h$.
Let $k$ be any  subfield of $\R$, including the possibilities
$k=\Q$ and $k=\R$. Then the powers $h^i$ for $i=0,\ldots,5$
form a basis for $k[P,Q][h]$ as a free module over $k[P,Q]$
and for $k(x,y)$ as a vector space over $k(P,Q)$.
So the field extension $k(F)=k(P,Q)\subset k(X)=k(x,y)$ is of degree $6$.

\begin{Prop}
Let $F_1=(P,Q_1)$ and $F_2=(P,Q_2)$ be \pk maps and let $Q_2=Q_1+S(P)$
for a \p $S$ in one variable with real coefficients.
Then $S$ is uniquely determined and its coefficients
belong to any  subfield of $\R$ that contains  the coefficients
of $Q_1$ and $Q_2$.
\end{Prop}

\begin{proof}
$P$ is transcendental over $\R$, so $S$ is unique.
Let $k$ be the subfield of $\R$ in question.
Let $c \in \Q$ with $c \ne 0$ and $c \ne -1$.
Choose $h \in \Q$ that is not a pole of the  previously described rational parametrization $x(h),y(h)$ of the level set $P=c$.
Since both $x(h)$ and $y(h)$ have formulas 
in $\Q(h,c)$, the real number
$S(c)=Q_2(x(h),y(h))-Q_1(x(h),y(h))$ actually is in $k$.
The coefficients of $S$ can be reconstructed, using rational arithmetic, from its values at any $j > \deg S$ such points $c$,
and so are in $k$.
\end{proof}

\begin{Cor}
Let $F_1$ and $F_2$ be \pk maps and let $k$ be $\R$ or any subfield of $\R$ over which both maps are defined.
Then $k(F_1)\subset k(X)$ and $k(F_2)\subset k(X)$
are one and the same field extension.
\end{Cor}

\begin{proof}
Since $S$ has coefficients in $k$, the relation
$Q_2=Q_1+S(P)$ implies that
$k(F_1)=k(P,Q_1)=k(P,Q_2)=k(F_2)$. 
\end{proof}

\begin{Thm}\label{extension}
Let $F$ be any \pk map and let $k$
be $\R$ or any subfield of $\R$
containing  the coefficients of $F$.
Although $F$ is generically two to one as a \p map of
$\R^2$ to $\R^2$, the degree of the associated 
extension of function fields over $k$ is six. 
Furthermore, the extension has no automorphisms 
other than the identity, and so, in particular, it is 
not {G}alois. 
\end{Thm}

\begin{proof}
The conclusions have already been drawn for the earlier  specific 
\pk map $F=(P,Q)$ and for 
$k=\R$ (Corollary \ref{six} and Proposition \ref{notg}). 
Both \pk maps have the same function \fe over $k$, so 
it has degree six. And any nontrivial automorphism 
defined over $k$ defines one over $\R$, 
since the $\R$-linearly extended 
automorphism  preserves $P$, $Q$, and $\R$. 
\end{proof}

\section{Real \J conjectures}\label{conj}

 The Strong Real \J Conjecture (SRJC), 
as formulated in the introduction,
asserted that a polynomial map $F: \R^n \To \R^n$,
has a real analytic  inverse if 
its \J determinant, $j(F)$,  vanishes nowhere on $\R^n$.
It was refuted by the {P}inchuk counterexamples, so only special cases are continuing subjects of inquiry. Both the hypothesis concerning $j(F)$ and the conclusion that a real analytic inverse exists can be restated in various equivalent ways. Principally, the former is equivalent to
the assertion that $F$ is locally diffeomorphic or locally real bianalytic, and the latter to the assertion that $F$ is injective (one-to-one) or bijective (one-to-one and onto) or a homeomorphism or a diffeomorphism. These are all obvious, except for the key result that injectivity, also called univalence, implies bijectivity
for maps of $\R^n$ to itself that are
polynomial or, more generally, rational and defined on
all of $\R^n$ \cite{InjectiveReal}.
That result does not generalize to semi-algebraic maps of $\R^n$ to itself \cite{SurjectiveSemialgebraic}.
To avoid any possible confusion,
an everywhere defined real rational map has
components that belong to $\R(\vars)$,
each of which can be written as a fraction
with a polynomial numerator and a
nowhere vanishing polynomial denominator. 
In addition to their global and local properties in the 
strong (Euclidean) topology, such maps are continuous in the Zariski topology. 
Let the (false) rational real \J conjecture (RRJC))) be the
extension of the SRJC to everywhere defined  rational maps.
Clearly any global univalence theorems \cite{GlobU}
that apply to locally diffeomorphic maps $\R^n \To \R^n$
yield (true) special cases of the conjecture. 
Properness suffices, and related topological considerations 
play a role below. 
But the focus of this article is on results or conjectures that require
the polynomial or rational character of a map
and involve properties of the associated algebraic field extension .

Remark.  The term 'rational real \J conjecture'
and the acronym RRJC are not standard terminology. 
Even if rational maps are allowed, 
the conjecture is usually called the real \J conjecture
or even just the \J conjecture, despite considerable ambiguity.
And the term 'rational' is used with various shades of meaning. 
In some work on the two dimensional complex JC 
(e.g., \cite{SimpleRationalPolys}), a rational polynomial
is a polynomial in two complex variables  whose generic fiber is 
a rational complex curve. 
And Vitushkin \cite{Vitushkin} has presented
$F= (x^2 y^6+2xy^2, xy^3+1/y)$
as a sort of rational counterexample to the JC.
$F$ has constant \J determinant $j(F)=-2$ and
$F(-3,-1) = F(1, 1) = (3, 2)$
So $F$ is not injective when considered as a rational map of $\R^2$ to itself. But $F$ is also not defined everywhere on $\R^2$.

In the RRJC context,
the distinction between \nz constant and nowhere vanishing \J 
determinants is not as critical as it may seem. 
If $F:\R^n \To \R^n$ satisfies the hypotheses,
let $x \in \R^n, z \in \R$ and define 
$F^+:\R^{n+1} \To \R^{n+1}$
by $F^+(x,z)=(F(x),z/(j(F)(x))))$.
Then $F^+$ also  satisfies the hypotheses, $j(F^+)=1$, and 
$F^+$ is injective if, and only if, $F$ is injective. 
As pointed out  in 
\cite{RealJC+SamuelsonMaps}, choosing {P}inchuk maps for $F$ yields 
counterexamples to the RRJC in dimension $n=3$ with
\nz constant \J determinant. 

Note that all three conjectures are true in the 
 dimension $n=1$ case $f: \R \To \R$. 
 In the JC case, $f$ is of degree $1$. In the SRJC 
case, $f$ is proper, since any nonconstant polynomial becomes infinite 
when its argument does. 
In the RRJC case, $f$ is monotone increasing or decreasing, 
hence injective, thus surjective, so unbounded above and 
below, and therefore proper.

Let $F=(\comps)$ be  a real rational map of $n$ real variables $\vars$ and $k$ a subfield of $\R$ such that each component 
$f_i$ belongs to $k(X)=k(\vars)$.
Whether defined on all of $\R^n$ or not, if $j(F)$ is not identically zero, then $k(F) = k(\comps)\subseteq k(X)$ is an algebraic field extension  of finite degree. 
The reasons are the same as in the polynomial case,
and the following standard lemma shows they are true in broader contexts having nothing to do with $\R$. 

\begin{Lemma}
Let $k$ be a field of characteristic zero and suppose that 
$\comps \in k(\vars)$ are algebraically dependent over $k$.
Let $j(F) \in k(\vars)$ be the \J determinant of $F=(\comps)$.
Then $j(F)=0$.
\end{Lemma}

\begin{proof}
Suppose, to the contrary, that $h(\comps)=0$ for a \nz polynomial 
$h \in k[y_1,\ldots,y_n]$, but $j(F)\ne 0$. Put $K=k(\vars)$ and 
observe that $J(F)$, the \J matrix of $F$, is an invertible 
matrix with entries in the field $K$. From the chain rule 
$v \cdot J(F) =0$, where $v$ is the row vector with components $v_i = (\partial h/\partial y_i)(\comps) \in K$.
It follows that $v=0$, and so each first order
partial derivative of $h$ is either the zero polynomial or it defines a new 
relation of algebraic dependence. By induction, the same is 
true for partials of all orders. But since $h$ is a polynomial, at least one partial of some
order is a \nz constant, yielding a contradiction. 
\end{proof}

So both the {G}alois and birational cases of the RRJC are meaningful. 
The {G}alois case of the standard JC is settled and true 
over any field $k$ of characteristic zero. It states that
a polynomial map $k^n \To k^n$ with constant \nz \J determinant 
and a {G}alois field extension  
has a polynomial inverse. It was first proved for $k=\C$ only \cite
{GaloisCase}, using methods of the theory of several complex variables. The general case appears in \cite{Razar} and, independently,
\cite{AlgebraicGaloisCase} .
For a recent proof, see \cite[Thm. 2.2.16]{ArnoBook}. Of course, the existence of 
a polynomial inverse implies the triviality of the field extension , so the theorem has 
no concrete examples. 
In contrast, in the SRJC and RRJC contexts, the existence of 
an inverse does not imply the field extension  is {G}alois, much less birational. 
For instance, if $y=f(x)=x^3+x$, the field extension  $\R(y) \subset \R(x)$ 
is neither.

The following are basic properties of an everywhere defined 
real rational map
$F: \R^n \To \R^n$ with $j(F)$ vanishing nowhere.
It is a local diffeomorphism, hence an open map.
Let $x \in \R^n$ and $y= F(x) \in \R^n$ and define $m(x)$
to be the number of inverse images of $y$ under $F$, 
potentially allowing $+\infty$ as a possible value. 
Since $F$ is open, $m(x')\ge m(x)$ for $x'\in \R^n$ in a neighborhood 
of $x$. So if $A \subseteq \R^n$, the maximum value of $m$ 
on $A$ is also the maximum value of $m$ on its topological 
closure $\bar{A}$. 
Let $t \in \R(X)$ be a primitive element (generator) for the field extension  
$\R(F) \subseteq \R(X)$, and suppose it satisfies an equation of
minimal degree $d$ in $t$ over $\R[F]$.
Temporarily define a Zariski open subset $U$ of $\R^n$ by declaring $x \in U$ If $y=F(x)$ is not a zero of 
the leading coefficient of the equation for $t$ 
and $x$ is not a zero of a selected specific  common denominator for $t$ and the coefficients of the expressions for the components of $F$ as polynomials in $t$. 
If $x \in U$, then $t(x)$ can have at most $d$ values
and so $m(x)$ is also not greater than $d$. 
Let $N \le d$ be the maximum value of $m(x)$ on $U$.  
Because $\bar{U}=\R^n$,  $N$ is in fact the maximum 
value of $m(x)$ for all $x \in \R^n$. 
That shows not only that $F$ is quasifinite, meaning that
the inverse image of any point in the codomain is a finite set, 
but also that the degree of the field extension  $\R(F) \subseteq \R(X)$ is a global bound on the size of those sets. 
All subsets of $\R^n$ that can be described in the first order logic 
of ordered fields are semi-algebraic. 
The description can include real constant symbols 
(coefficients, values, etc.) and quantification over real 
variables (but not over subsets, functions or natural numbers); 
results for any dimension $n>0$ and involving 
polynomials of arbitrary degrees follow from 
schemas specifying first order descriptions for any fixed 
choice of the natural number parameters. 
As a first application of that principle,  the $N$ subsets of the 
domain $\R^n$ on which $m(x)$ has a 
specified numeric value in the range $1,\ldots,N$,
and the  $N+1$ subsets of the codomain $\R^n$ on which $y$ has a 
specified number of inverse images in the range $0,\ldots,N$, 
are all semi-algebraic. 
The set of points $y$ in the codomain at which $F$ is proper is 
readily verified to be the open set of points at which the number of 
inverse images of $y$ is locally constant. 
That set 
contains all points with $N$ inverse images and has an $\epsilon$-ball
first order description.
Its complement $A(F)$, the asymptotic variety of $F$, is therefore 
closed semi-algebraic and the inclusion $A(F)\subset \R^n$ is strict. 
$A(F)$ Is the union for $i=0,\ldots,N-1$ of the  semi-algebraic sets consisting of points $y$ in the codomain at which $F$ 
is not proper and for which $y$ has exactly $i$ inverse images. 
At an interior point $y$ of one of these sets $F$ would be proper, 
contradicting $y \in A(F)$. Thus each such set has empty interior, 
hence is of dimension less than $n$. 
Consequently $\dim A(F) < n$. 
It follows that the complement of $A(F)$ is a finite union of disjoint 
connected open semi-algebraic subsets of $\R^n$ on each of which the number of 
inverse images of points is a constant, with possibly differing 
constants for different connected components.
If $U$ is such a connected component, then $F^{-1}(U)$ is 
also open and semi-algebraic. If it is not empty, let $V$ be one of its 
finitely many connected components. Since $V$ is an open and closed subset of $F^{-1}(U)$, the map $V \To U$ induced 
by $F$ is a proper local homeomorphism of connected, locally compact, and locally arcwise connected spaces and hence it is a covering map. Such a map is surjective, so all of $U$ is 
contained in $F(\R^n)$. 
$V$ must be exactly one of the finitely many connected components 
of the open semi-algebraic set $\R^n \setminus F^{-1}(A(F))$, 
since it is closed in that subset as one element of a finite cover by 
total spaces of covering maps. 
Speaking informally, this presents a view of $F$ as a finite 
collection of $n$-dimensional covering maps, of possibly 
different degrees, glued together along semi-algebraic sets of 
positive codimension to form $\R^n$ at the total space level, 
whose base spaces, which may sometimes coincide for different 
total spaces, are similarly glued together to form $F(\R^n)$. 
While $F(\R^n)$ is open and connected, it
would not be dense in $\R^n$ if $F^{-1}(U)=\emptyset$ for some connected component $U$ of $\R^n \setminus A(F)$,
a possibility that has not been ruled out. 
$F(\R^n) \cap A(F)$ is in general neither empty nor all of $A(F)$, 
a behavior exhibited by any {P}inchuk map.

Remark. There is an extensive body of work by Zbigniew Jelonek 
defining, investigating, or related to the concept of the set of 
points at which a polynomial map is not proper. 
In \cite{geometry} he covers and sharpens the just described facts, at least for polynomial maps. As one result, he shows that 
for a nonconstant polynomial map $F: \R^n \To \R^m$, where $n$ and $m$ 
are any positive integers and no other conditions are imposed, 
 the set $A(F)$ of points at which $F$ is not proper 
is $\R$-uniruled. By that he means that for any $a \in A(F)$ there is 
a nonconstant polynomial map $g: \R \To \R^m$ (a polynomial curve) such that 
$g(0)=a$ and $g(t) \in A(F)$ for all $t \in \R$. That in turn 
implies that every connected component of $A(F)$ is unbounded 
and has positive dimension. 
In the same article, Jelonek explicitly considers the SRJC and 
shows, using topological methods, that $F$ has an inverse 
(and hence $A(F)=\emptyset$)  if 
$A(F)$ has codimension three or higher. 
That and related results will be reconsidered below in the RRJC context.

\begin{Ex}\label{uni}
Let $f:\R \To \R$ be the real rational map given by 
$y=1/(x^2+1)$. The image $f(\R)$ is the half open interval 
$0 < y \le 1$. The point $y=0$ is the only point at which $f$ 
is not proper. So $A(f)=\{0\}$ is of dimension $0$ and not 
uniruled. 
\end{Ex}

\subsection{Points of definition}\label{conj1}{\ \newline}

This section states  a number of standard definitions and 
assembles some associated technical results for later use. 

Let $k$ be any field. Characteristic zero is not assumed. 
It is not technically correct to speak of {\it the coefficients} 
of a rational function, since it is an equivalence  class of 
(numerator, denominator) pairs. A rational function is said to be 
{\it defined over $k$} if it has a representative pair
with coefficients in $k$. And a rational function  of $n$ variables defined over $k$ 
is said to be {\it defined at a point} $x \in k^n$ if the denominator 
can also be chosen so that it is not zero at $x$.
Eliminating common factors of the numerator and denominator 
yields a {\it reduced fraction}, and unique factorization 
shows that all reduced fraction representations of a given 
rational function have the same numerator and the same 
denominator, up to multiplication by nonzero elements of $k$; 
such a denominator  is zero exactly at the 
points at which the function is not defined. 
Any reduced fraction for $f \in k(X)$ 
can be used to determine if $f$ is defined at a point $x \in k^n$, 
and if so, to compute its value $f(x)$; this also applies if $x$ is 
allowed to have coordinates in a commutative $k$-algebra.

Let $K$ be any field containing $k$ as a subfield.

\begin{Lemma}\label{a=bc}
Let $a,b,c \in K[X]$ satisfy $a=bc \ne 0$. If any two of them 
belong to $k[X]$, so does the third.
\end{Lemma}

\begin{proof}
Clear if $b,c \in k[X]$. Suppose, without loss of 
generality, that $a,c \in k[X]$.
Choose a term order for $\vars$. That is, choose a total order 
on monomials in $\vars$ compatible with multiplication. 
Comparison of leading terms reveals that $b$ has as leading term 
$t$ the product of a monomial and a coefficient in $k$. 
Let $a'=a-tc=(b-t)c$. 
If $a'=0$, then $b=t \in k[X]$.
If not, conclude by descending induction on 
the maximum order of terms in $b$. 
\end{proof}

\begin{Lemma}\label{k(t)}
If $K/k$ is purely transcendental and $p \in k[X]$ is
 irreducible, , then $p$ remains irreducible in $K[X]$.
\end{Lemma}

\begin{proof}
If $p=ab$, then the factorization actually takes place in $k[X]$.
To prove that, first, assume a simple transcendental extension $K=k(t)$. Then both factors must have degree zero in $t$.
Next, induction handles finite transcendence degree.
Finally, a counterexample could only involve finitely many 
elements of a transcendence base.
\end{proof}

\begin{Lemma}\label{Bass}
If $K/k$ is algebraic and $p,q \in k[X]$ are relatively prime, 
then they remain  relatively prime in $K[X]$.
\end{Lemma}

Remark. The following proof was privately communicated by Hyman Bass.

\begin{proof}
Let $A=k[X]$ and $B=K[X]$. 
As an extension of commutative rings, $B/A$ is integral, 
because $B$ is generated over $A$ by 
the elements of $K$, which are algebraic 
over $k$ and hence integral over $A$. 
Let $J_k$ ($J_K$) be the ideal generated by $p$ and $q$ 
in $A$ (respectively, $B$). If $p$ and $q$ have an  
irreducible common divisor in $B$, it generates a 
height $1$ prime ideal $I$, such that $J_K \subseteq I$. 
Contraction preserves height for integral extensions, 
so $I \cap A$ is a height $1$ prime ideal in $A$. 
As $A$ is the polynomial algebra $k[X]$, a height $1$ prime ideal 
is a principal ideal generated by an irreducible polynomial. 
Since $J_k \subseteq I \cap A$, that polynomial divides both 
$p$ and $q$ in $A=k[X]$, contradicting the 
hypothesis that they are relatively prime.  
\end{proof}

\begin{Lemma}\label{redfrac}
Reduced fractions remain reduced for any \fe $K/k$. 
\end{Lemma}

\begin{proof}
If $K/k$ is algebraic, this is just a restatement of 
Lemma \ref{Bass}. Suppose now that $K/k$ is pure 
transcendental. Factor numerator and denominator into 
irreducible polynomials in $k[X]$. In $K[X]$ the polynomials 
are irreducible by Lemma \ref{k(t)} and any potential 
cancellation has quotient in $k$ by Lemma \ref{a=bc}, thus 
contradicting the hypothesis that the original fraction 
is reduced. 
The general case follows immediately, since any \fe  is an 
algebraic extension of a pure transcendental extension. 
\end{proof}

\begin{Ex}
The real polynomial fraction $1/(x^2 + 1)$ is defined 
over $\Q$ and reduced over $\C$. 
\end{Ex}

\begin{Thm}\label{Consistency} (Consistency)
Let $f \in K(X)$ be a rational function in $n>0$ 
variables defined over a field $K$. Suppose that $f$ is 
defined over a subfield $k \subset K$. 
Then
\begin{enumerate}
\item $f=p/q$, where $p$ and $q$ are relatively prime
polynomials in $K[X]$ that have coefficients in $k$. 
\item $f$ is defined at $x \in K^n$ if, and only if, $q(x) \ne 0$. 
\item $f(x)$ is defined if, and only if, it is defined viewing $k$ 
as the coefficient field and $K$ as a commutative $k$-algebra; 
if so, the value is the same. 
\item if defined, $y=f(x) \in K$ belongs to the subfield 
generated by $k$ and the coordinates of $x$. 
\end{enumerate}
\end{Thm}

\begin{proof}
To prove the first conclusion, express $f$ as a reduced fraction 
$p/q$ with $p,q \in k[X]$, then apply Lemma \ref{redfrac}. 
The second is an earlier noted property of reduced fractions,
easily proved by unique factorization. 
For the third, use $p/q$ in both cases and observe that the 
test and conditional value computation are identical. 
Fourth, $p(x)$ and $q(x)$ obviously belong to that subfield. 
\end{proof}

Remark. Although more cumbersome to state, the results 
for $x \in K^n$ extend to points with coordinates in 
a commutative $K$-algebra.

The most important idea in the above theorem is that, 
although polynomials can acquire new factorizations when the 
coefficient field is extended, rational functions cannot 
acquire new value definitions. 
That is, a rational function 
either remains undefined at a point  or retains its prior value, 
no matter how it is expressed or simplified in the new context.

Over $\R$, there is an entirely different  way, involving analytic 
functions, to determine if a rational function is defined at a point.
Let $U$ be an open subset of $\R^n$ and $f$ a real analytic
function defined on $U$. If $x \in \R^n$ is a point on the 
boundary of $U$, then $f$ is said to {\it extend analytically} to $x$ 
if there exists a real analytic function $g$, defined on an open 
neighborhood $V$ of $x$, such that $f=g$ on $U \cap V$.

\begin{Lemma}\label{Analytic}
Let $f \in \R(X)$ be a real rational function and write it as 
$f=p/q$, where $p$ and $q$ are real polynomials with no 
common nonconstant real polynomial divisor. 
As a real valued function, $f$ is well defined and real analytic 
on the open set $U$ where $q \ne 0$. 
If $x \in \R^n$ and $q(x)=0$, 
then $x$ is a boundary point of $U$, 
but $f$ cannot be analytically extended to $x$.
\end{Lemma}

\begin{proof}
All clear, except the issue of extending $f$ to $x$. 
Let $g$ be a real analytic extension to $x$. Let $B$ be a small 
Euclidean ball around $x$ in $\C^n$ on which the power series 
expansion for $g$ at $x$ converges absolutely and uniformly, 
defining a complex analytic function $\tilde{g}$ on $B$ that 
restricts to $g$ on $B'=B \cap \R^n$. Since $p=gq$ on 
$B' \cap U$, the same relation holds for the power series 
expansions at $x$. So $p=\tilde{g}q$ on $B$. 
By Lemma \ref{Bass}, $p$ and $q$ are relatively prime as complex
polynomials. That can also be proved, somewhat more simply, 
by using complex conjugation. 
So the complex hypersurfaces $p=0$ and $q=0$ 
have no common irreducible components. Since complex hypersurfaces, 
unlike real hypersurfaces, cannot have isolated points, 
there exist points $x'\in B$ arbitrarily close to $x$ 
that satisfy $q(x')=0$ and $p(x')\ne 0$. That contradicts 
$p=\tilde{g}q$. Even if attention is restricted to points at 
which $q$ is nonzero, the values of $\tilde{g}=p/q$ 
on such points would 
not be bounded in any neighborhood of $x$. That contradicts 
the analyticity of $\tilde{g}$, proving that the presumed analytic 
function $g$ cannot exist. 
\end{proof}

\begin{Ex}
The real rational function $f=(x^4 + y^4)/(x^2 + y^2)$ 
is not defined at the origin $(0,0)$. 
At every other point of the \xyp  it is defined and satisfies 
$0 < f(x,y) \le x^2 + y^2$, So setting $f(0,0)=0$ yields 
a unique continuous extension of $f$ to the entire plane. 
Although that extension is locally bounded at the origin, 
it is not real analytic there. 
\end{Ex}

\subsection{The birational case}\label{conj2}{\ \newline}

In this section a prefix, such as $k$-, will signal the specific field 
(or ring) of coefficients under consideration. The prefix will be 
omitted if clear from the context or irrelevant, and the terms 
'real' and 'complex' will usually be used instead of $\R$- and $\C$-.

\begin{Prop}\label{anyK}
Let $K$ be a field and $F: K^n \To K^n$ a rational map.
Suppose $F$ is defined over a subfield $k \subset K$. 
Then $F$ is $k$-birational ($k(F)=k(X)$) if, and only if, 
it is $K$-birational ($K(F)=K(X)$). 
\end{Prop}

\begin{proof}
The birationality condition $k(F)=k(X)$ is equivalent to 
the assertion that $k(X)$ has dimension $1$ as a vector space
over $k(F)$.
But that dimension is invariant under faithfully flat extension, 
as when tensoring over $k$ with any \fe $K$. 
\end{proof}

So if $k \subseteq \R$, then $k$-birationality is also equivalent to $\R(F)=\R(X)$ or  $\C(F)=\C(X)$. That is, 
if $F$ has a purely algebraic rational inverse 
with complex coefficients
(ignoring any issues about where the inverse is defined in 
$\C^n$ or in $\R^n$), then it has one
with real coefficients that belong to $k$.

\begin{Lemma}\label{restrict}
Let $K$ be a field and $f: K^n \To K$ an everywhere defined  rational function. Suppose $f$ is defined over a subfield $k \subset K$. 
Then $f$ restricts to an everywhere defined  $k$-rational function from $k^n$ to $k$. 
\end{Lemma}

\begin{proof}
Considering $f$ as a  $k$-rational function, the 
consistency assertions of Theorem \ref{Consistency} imply 
that its domain of definition is $K^n \cap k^n = k^n$ and that its 
values there (necessarily in $k$) are those it had as a $K$-rational function. 
\end{proof} 

A similar conclusion applies to an everywhere defined map with any 
finite number of rational function components. While a rational 
bijection $K^n \To K^n$ defined over a subfield $k \subset K$ will 
restrict to an injection $k^n \To k^n$, that map does not have to 
be surjective. 

\begin{Ex}\label{x+x^3}
Let $f:\R \To \R$ be given by $y=f(x)=x+x^3$. The 
restriction of $f$ to $\Q$ is a map into, but not onto, $\Q$. 
For instance, $y=1$ is not a value by the rational root test. 
\end{Ex}

\begin{Lemma}\label{injective}
Let $F: \R^n \To \R^n$ be a continuous  
open semi-algebraic map. 
If $F$ is injective on a Zariski open subset of $\R^n$,
 then it is injective on all of $\R^n$. 
\end{Lemma}

\begin{proof}
Let $U$ be a  Zariski open subset of $\R^n$, such that $F$ 
is injective on $U$.  
The complement of $U$ in $\R^n$ is semi-algebraic (indeed algebraic), 
of maximum dimension at most $n-1$. By the general form of 
the Tarski-Seidenberg projection property, its image under $F$ 
is also semi-algebraic of maximum dimension at most $n-1$. 
So it is not Zariski dense. Let $Z$ be its Zariski closure. 
The set $F^{-1}(Z)$ is semi-algebraic. It also has empty interior, 
as otherwise $Z$ would contain an open set. So it has maximum 
dimension at most $n-1$, and hence is not Zariski dense. 
Let $Z'$ be its Zariski closure, and let $U'=U\setminus Z'$. 
Then $U'$ is a nonempty Zariski open subset of $\R^n$ and for 
every $x \in U'$, the point $y=F(x)$ has only one inverse 
image anywhere in $\R^n$. 
If  $F$ is not injective, let $a,b \in \R^n$ satisfy $a \ne b$ 
and $F(a)=F(b)$. 
Take disjoint open sets $U_a$ and $U_b$, such that 
$a \in U_a$ and $b \in U_b$. 
Since $F(U_a) \cap F(U_b)$ is open and $F$ is continuous,
one can shrink $U_a$ and $U_b$ so that they also satisfy
$F(U_a)=F(U_b)$. So for any $x \in U_a$, the point $y=F(x)$ 
has at least two inverse images. Since $U'$ is Zariski open, 
$U_a \cap U'\ne \emptyset$, and any point of intersection 
yields a contradiction. 
\end{proof}

Remark. The final part of the above proof is a specific case of arguments about the counting function $m(x)$ in section \ref{conj}.

\begin{Lemma}\label{inj}
Let $F: \R^n \To \R^n$ be an everywhere defined real rational  
map. If $F$ is both birational and an open map, then 
$F$ is injective. 
\end{Lemma}

\begin{proof}
By birationality, there exist Zariski open subsets $U$ and $V$ of 
$\R^n$, such that $F$ maps $U$ bijectively onto $V$. 
As $F$ is injective on $U$, it satisfies the hypotheses of 
Lemma \ref{injective}, so is injective on $\R^n$. 
\end{proof}

\begin{Thm}\label{bicase} ([Birational case of the RRJC)
Let $F: \R^n \To \R^n$ be an everywhere defined real rational  
map with nowhere vanishing \J determinant.
If $F$ is birational, then it has 
an everywhere defined real rational  inverse.
And in that case, if $F$ is defined over a subfield $k \subset \R$, 
then its restriction to $k^n$ is a $k$-rational everywhere 
defined bijection of $k^n$ onto $k^n$, and that also holds for 
its inverse.
\end{Thm}

\begin{proof}
$F$ is an open map, so it is injective by Lemma \ref{inj}.
Hence it is surjective by the  Bia{\l}ynicki-Birula and Rosenlicht Theorem \cite{InjectiveReal}. 
Since it is locally bianalytic, $F$ has a global inverse, 
$F^{-1}$, that is real analytic on all of $\R^n$. 
$F^{-1}$ is a real analytic extension to $\R^n$ of the 
rational inverse of $F$. By Lemma \ref{Analytic}, each 
component of $F^{-1}$ is an everywhere defined real rational function 
and so $F^{-1}$ is an everywhere defined real rational map. 
If $F$ is defined over a subfield $k \subset \R$, 
start with a rational inverse with coefficients in $k$, as is possible 
by Proposition \ref{anyK}. 
Argue as before, then apply Lemma \ref{restrict} componentwise 
to both $F$ and $F^{-1}$. 
\end{proof}

Remark. In \cite{PolynomialRational}, \p maps $F:\R^n \To \R^n$ 
that map $\R^n$ bijectively onto $\R^n$ are considered, and the question is raised 
of when the inverse is rational. If so, the inverse is everywhere 
defined on $\R^n$ and $F$ is called a polynomial-rational bijection (PRB) of $\R^n$. 
A key technical result is that a \p bijection is a PRB if its 
natural extension to a \p map $\C^n \To \C^n$ maps only 
real points to real points. A PRB $F$ has a nowhere vanishing \J determinant $j(F)$. Conversely, it is shown that a nowhere 
vanishing $j(F)$ alone suffices to establish that a \p map 
$F:\R^n \To \R^n$ of degree two is a bijection and a PRB. 
A related but stronger condition is defined and shown to be sufficient, but not necessary, for \p maps of degree greater than two.

\subsection{Promoted {SRJC} cases}\label{conj3}{\ \newline}

This section is primarily concerned with some known special  cases 
in which the SRJC holds on the basis of topological considerations 
implying injectivity, and which therefore generalize effortlessly to 
the RRJC context. The special cases appear in the previously cited 
paper \cite{geometry} by Zbigniev Jelonek and in a fairly recent 
note \cite{NewProperness} by Christopher I. Byrnes and Anders Lindquist. 

Let $F:X \To Y$ be a continuous map of topological  manifolds.  Let $A(F)$ be the set of points $y \in Y$ at which $F$ is not proper, and let 
$B(F) = F^{-1}(A(F))$. Recall that $A(F)$ is closed, that the 
restriction of $F$ to the induced map $X \setminus B(F) \To Y \setminus A(F)$ is proper, and that $A(F)$ is the smallest subset of $Y$ with those properties. If $F(X)$ is open, then its topological boundary $\partial F(X)$ is contains in $A(F)$. In nice enough cases, removing subsets of codimension $i$ does not affect homotopy groups in dimensions less than $i-1$.  Indeed, in  \cite[Lemma 8.1]{geometry}, Jelonek proves that if $A$ is a closed semi-algebraic subset of $\R^n$, then $\R^n \setminus A$ is connected if $A$ is of codimension greater than one and simply connected if it is of codimension greater than two. The usual conventions  apply, namely that the empty set has dimension $-\infty$ and 
codimension $+\infty$.

\begin{Thm}\label{promoted}
Let $F:\R^n \To \R^n$ be a real rational everywhere defined map 
with nowhere vanishing \J determinant. 
Let $A(F)$ be the set of points in the codomain at which $F$ is 
not proper. Then the following are equivalent:
\begin{enumerate}
\item $F$ has a global real analytic inverse, 
\item $A(F)=\emptyset$,
\item $\dim(A(F))  < n-2$,
\item  $A(F) \cap F(\R^n)=\emptyset$,
\item $A(F)=\partial F(\R^n)$.
\end{enumerate} 
\end{Thm}

\begin{proof}
The well known so-called topological Hadamard theorem 
states that a local homeomorphism $\R^n \To \R^n$ is a 
homeomorphism if, and only if, it is proper. That yields the 
equivalence of (1) and (2) in the current context, since the 
\J condition implies that $F$ is locally real bianalytic. 
The equivalence therefore does not require the
 surjectivity theorem (ST) of \cite{InjectiveReal} 
for injective rational maps  $\R^n \To \R^n$. The essential points are that a proper local homeomorphism 
of connected manifolds (necessarily of the same dimension) is a 
covering map (necessarily surjective) and that the base, $\R^n$, 
is simply connected and so has no nontrivial connected cover. 

Obviously, (2) implies (3) and (4). It implies (5) as well, 
because (1) implies that $\partial F(\R^n)=\emptyset$.

In case (3), let $B(F) = F^{-1}(A(F))$. The induced map 
$\R^n \setminus B(F) \To \R^n \setminus A(F)$ is a proper 
local homeomorphism. 
Since $\dim(A(F)) < n-2$, the base is simply connected. 
Because $F$ is a  local homeomorphism, $\dim(B(F)) < n-2$. 
So the total space is certainly connected. It follows that $F$ is 
injective on $\R^n \setminus B(F)$. $B(F)$ is not Zariski dense, 
so $F$ is injective on a Zariski open subset of $\R^n$. 
By Lemma \ref{injective} it is injective on $\R^n$. 
Finally, use the ST to conclude that $F$ is also surjective and 
therefore (1) holds. The result that (3) implies (1) in the SRJC context (that is, for polynomial maps with nowhere vanishing \J determinant) 
is precisely what is proved in \cite[Thm 8.2]{geometry}, and 
Jelonek's proof is the model for  the one presented here.

In case (4), since $A(F)$ is contained in   the closure of 
the image of $F$, the condition $A(F) \cap F(\R^n)=\emptyset$
amounts to saying that the map of  $\R^n$ onto its image 
is proper.
The main result of \cite{NewProperness} is that the standard complex 
JC holds for polynomial maps that are proper as maps onto their image. 
In a remark at the end of the note, (4) is proved to imply (1) 
in the SRJC context. Briefly, (4) implies that $\R^n$ is a 
universal covering space, of finite degree $d$, of $F(\R^n)$. 
By well known results of the branch of topology called P. A. Smith 
theory, there are no fixed point free homeomorphisms of $\R^n$ 
onto itself of prime period. But the fundamental group 
$\pi_1 (F(\R^n))$ is of order $d$, and contains an element of 
prime period unless $d=1$. So $d=1$, $F$ is injective, and (1) 
follows as in case (3), by using the ST. The assumption that $F$ 
is polynomial, rather than just real analytic, is used at only two points in 
the proof. First, it ensures that the degree of the covering map is 
finite. Second, it allows the ST to be applied. Rationality is 
sufficient in both situations, so (4) implies (1) in the RRJC 
context as well.

Case (5) is equivalent to case (4) because $F(\R^n)$ is open in 
$\R^n$, hence $\partial F(\R^n) \subseteq A(F)$ and $A(F)$ is 
contained in the disjoint union of $F(\R^n)$ and 
$\partial F(\R^n)$. 
\end{proof}

 Apart from the two special cases above, there are some closely related issues worth noting. Continue to assume that 
$F:\R^n \To \R^n$ is a real rational everywhere defined map 
with nowhere vanishing \J determinant, and, recall the general 
properties of such maps described in section \ref{conj}. 
$A(F)$ has 
positive codimension and $\R^n \setminus A(F)$ is the 
disjoint union of finitely many connected open subsets of $\R^n$, 
each of which is either entirely contained in the image of $F$ or has empty inverse image. 
Clearly $F(\R^n)$ is dense in $\R^n$ if, and only if, none of them 
has an empty inverse image. 
If $\dim(A(F))<n-1$ then $\R^n \setminus A(F)$ is connected, 
so it has only a single connected component and $F$ has dense 
image. 
The codimension at least two condition (CD2) is of particular interest for 
dimension $n=2$. In that case $A(F)$ is either empty or a finite 
set of points. If $F$ is polynomial, then $A(F)$ is $\R$-uniruled, only 
$A(F)=\emptyset$ is possible, and so the SRJC holds if CD2 and 
$n=2$ are true \cite[Section 8]{geometry}. That line of 
reasoning is not available for rational maps.

The condition that the image of $F$ is dense (DI) is strictly weaker 
than CD2, as shown by considering  {P}inchuk  maps. 
But even so,  it has important implications for the coimage,   $\R^n \setminus F(\R^n)$ of $F$. The coimage is always closed and semi-algebraic. If DI is true, then 
each connected component of, and hence all of, 
 the complement of $A(F)$ is contained in the image of $F$. 
Equivalently, the coimage of $F$ is contained in $A(F)$. 
Since $A(F)$ has codimensiom at least one, so does the coimage. 
No example is known for which DI is false, in which case the 
coimage would have codimension zero. Of course, if the \J 
condition is dropped, there are examples aplenty, such as 
$y=f(x)=x^2$. Combining several results yields

\begin{Thm}\label{DI}
Let $F:\R^n \To \R^n$ be a real rational everywhere defined map 
with nowhere vanishing \J determinant. 
Let $A(F)$ be the set of points in the codomain at which $F$ is 
not proper, and let 
$B(F) = F^{-1}(A(F))$.
Then $F$ has dense image if, and only if, the coimage of 
$F$ is contained in $A(F)$. And in that case, $A(F)$ is the 
disjoint union of the coimage and $F(B(F))$. 
If $F$  has dense image, but is not surjective, then 
the coimage and $F(B(F))$ are both nonempty, so the  coimage 
is a nonempty, closed, semi-algebraic, proper subset of $A(F)$. 
\end{Thm}

\begin{proof}
If $F$ has dense image, then the preceding paragraph shows that the coimage, which is always closed and semi-algebraic, is contained in $A(F)$. 
The converse is clear. 
It is nonempty if $F$ is not surjective. If it were all of $A(F)$, then $A(F)$ would be disjoint from the image of $F$, and so $F$ would have an inverse by Theorem \ref{promoted}; and hence be 
surjective. If $y \in A(F)$, then it is either in the image of 
$F$, and so is in $F(B(F))$, or it is not, and so is in the 
coimage. 
\end{proof}

In the complex JC context, it is well known that the coimage 
has complex codimension at least two. Briefly, the reasoning is as follows.  Since the coimage is closed and constructible, if it has
codimension less than two it  
contains an irreducible hypersurface $h=0$, $h \circ F$ vanishes 
nowhere and so is constant, contradicting the algebraic 
independence of the components of $F$. 
In the SRJC and RRJC contexts, there are no parallel  results for 
the real codimension of the coimage, even if the map has dense 
image.

\subsection{Dense images}\label{conj4}{\ \newline}

Let $F:\R^n \To \R^n$ be an everywhere defined real rational map. 
There is a simple algebraic criterion that guarantees that $F$ has a dense image, and it works even with a weaker \J condition. In this section, drop the 
assumption that $j(F)$ vanishes nowhere on $\R^n$, but do 
require that $j(F)$ is not identically zero on $\R^n$. 
As shown in section \ref{conj}, that is enough to ensure that 
the components of $F$ are algebraically independent and hence
$\R(X)$ is an algebraic extension of $\R(F)$ of finite degree. 
To facilitate a geometric view of this situation , introduce 
coordinates $y_1,\ldots,y_n$ in the codomain of $F$ and think of 
the map as given by $y_i=f_i(\vars)$, identifying 
$\R(y_1,\ldots,y_n)=\R(Y)$ with $\R(F)$. 
Let $d$ be the degree of the field extension, and $t \in \R(X)$ a primitive 
element (generator) over $\R(Y)$. 
Then $t$ is a root in $\R(X)$ of a degree $d$ irreducible polynomial 
$R(T)$ with coefficients in the polynomial ring $\R[Y]$, such that 
no nonconstant polynomial in $\R[Y]$ is a common divisor of all the 
coefficients. 
$R(T)$ is unique up to a nonzero real constant factor, and is
obtained from the monic minimal 
polynomial of $t$ over $\R(Y)$ by
writing its coefficients as reduced fractions and then
multiplying the whole polynomial by a least common multiple 
of the denominators. 
Write $R(y)(T)$ for the polynomial with real coefficients obtained by 
evaluating the coefficients at a point $y \in \R^n$. 
$R$ is also irreducible in $\R[Y,T]$, and so 
its set of zeros is an affine irreducible variety $V$ in $\R^{n+1}$. 
Use $y_1,\ldots,y_n,z$ as coordinates in $\R^{n+1}$. 
Clearly, $F$ factors as the rational map $(F,t):\R^n \To V$ 
followed by the projection $p_n:V \To \R^n$ onto the first $n$ components. 
$(F,t)$ is actually birational (by construction), but not necessarily everywhere defined, because 
$z=t(x)$ may not be defined at all points $x \in \R^n$. 
The projection map $p_n:V \To \R^n$ is regular, 
as it corresponds to the algebra homomorphism  $\R[Y] \subseteq \R[V] = \R[Y,T]/(R(T))$. 
The discriminant $D$ of $R(T)$ lies in the 
coefficient ring $\R[Y]$. 
Up to a nonzero factor in $\R(Y)$, it is 
the product of the squares of the differences of the 
roots of $R(T)$ in a splitting field. 
By {G}alois theory all the roots are primitive elements. 
The derivative of $R(T)$ with respect to $T$, 
a polynomial of degree $d-1$ in $T$,  would 
be zero at a repeated root. So all the roots are simple, and 
so $D$ is nonzero. 
There is a universal formula for the discriminant of a 
polynomial in one variable of a given fixed degree 
in terms of its coefficients. 
However, the formula applies only if the degree is actual, 
not just formal; that is, the leading coefficient must be 
nonzero. 
So $D(y)$ is the discriminant of $R(y)(T)$, 
provided that $y$ is a point at which the leading 
coefficient of $R(T)$ does not vanish.

\begin{Lemma}\label{roots}
There is a nonempty Zariski open subset $U$ of $\R^n$, such that 
for each $y \in U$ all the following hold:
\begin{enumerate}
\item $R(y)(T)$ has degree $d$,
\item the roots, real or complex, of $R(y)(T)$ are distinct 
($D(y)\ne 0$),
\item the number of inverse images of $y$ under $F$ equals 
the number of real roots of $R(y)(T)$,
\item for each $x \in R^n$ with $y=F(x)$, $t(x)$ is defined, 
and it is a different real root of $R(y)(T)$ for each different $x$. 
\end{enumerate}
\end{Lemma}

\begin{proof}
The function field of the 
variety $V \subset \R^{n+1}$ is $\R(Y)[t]=\R(F)[t]$, so 
$(F,t)$ is birational. Let $A$ and $B$ be Zariski open subsets 
of $\R^n$ and $V$, respectively, such that $(F,t)$ is a 
biregular map of $A$ onto $B$. The image of $\R^n \setminus A$ 
under $F$ and of $V \setminus B$ under $p_n$ are both \sa 
subsets of $\R^n$ of maximum dimension at most $n-1$. 
The Zariski closure of their union is therefore an algebraic set of  maximum dimension at most $n-1$. 
Let $U$ be its complement. Then $U$ is nonempty, 
Zariski open, 
$F^{-1}(U) \subseteq A$, and $p_n^{-1}(U) \subseteq B$. 
$U$ will be modified in the course of this proof. First, 
further restrict $U$ by requiring that both the leading coefficient 
and discriminant of $R(T)$ not have any zeros on $U$. 
That takes care of (1) and (2). 
If $y \in U$ and $z$ is a real root of $R(y)(T)$, then $(y,z)$ 
is a point of $V$ that lies in $B$. So $b=(y,z)$ is the image of 
a point $x \in A$ under $(F,t)$. Since $(F,t)$ is regular on $A$, 
this implies that $t$ is defined at $x$ and $t(x)=z$. 
That shows that the number of inverse images $x$ is at least 
as large as the number of real roots $z$. 
As $F^{-1}(U)\subseteq A$, $t$ is defined at any inverse image 
$x \in \R^n$ and $z=t(x)$ is a surjective map of inverse 
images to real roots. 
The final step in the proof is to further restrict $U$ so as to 
ensure that the correspondence is bijective. The $n$ coordinate 
polynomials $x_i \in \R(X)=\R(Y)[t]$ are each (uniquely) 
polynomials of degree less than $d$ in $t$ with coefficients 
in $\R(Y)$. Restrict $U$ to contain only points at which all 
the coefficients of those polynomials are defined. Then for 
$y \in U$, not only are $y=F(x)$ and $z=t(x)$ functions 
of $x$, but also $x$ is  a function of $y$ and $z$. 
\end{proof}

This leads directly to the following theorem. 
Note that the hypotheses on $F$ imply not only that 
the function field extension is algebraic of finite degree, but 
also that $F$ is bianalytic at some point, ensuring that 
the image of $F$ contains an open set and is thus at least 
Zariski dense. The theorem is moderately practical in 
application, allowing one to deal with a single specific 
polynomial in one variable.

\begin{Thm}\label{Dense}
Let $F:\R^n \To \R^n$ be an everywhere defined real rational map. 
Assume that the \J determinant of $F$  is not identically zero on $\R^n$. 
Let $t$ be a primitive element for the associated function field extension, and $m$ its monic minimal polynomial. 
Then the following are equivalent:
\begin{enumerate}
\item $F(\R^n)$ is dense in $\R^n$,
\item $F(\R^n)$ contains a Zariski open subset of $\R^n$,
\item $m$ is defined and has at least one real root 
on a Zariski open subset of $\R^n$, 
\item $m$ has at least one real root everywhere it is defined. 
\end{enumerate}
\end{Thm}

\begin{proof}
The polynomial $R(T)\in \R(Y)[T]$ of the preceding lemma, 
divided by its leading coefficient, is the monic minimal polynomial of $t$.
The latter is defined exactly at the points $y \in \R^n$ at which 
$R(y)(T)$ has full degree, because there is no common divisor 
of all the  coefficients of $R(T)$.  
And at those points both polynomials have the same roots. 
Since $m$ is defined everywhere except at the zeros of the 
leading coefficient of $R(T)$, (4) implies (3). 
If (3) holds for a Zariski open $V \subseteq \R^n$, 
the preceding  lemma implies that  the 
Zariski open subset $U \cap V$ is contained in the image of $F$, 
proving (2). Obviously, (2) implies (1). The final step is to 
show that (1) implies (4). Assume (4) is false. Take $y_0 \in \R^n$ 
at which $m$ is defined, but none of the roots are real. 
As long as the degree remains constant, the roots of a 
polynomial, here $R(y_0)(T)$,
depend continuously on its coefficients. 
 So there is an open neighborhood of $y_0$ on which $m$ is 
defined and has only complex roots. 
That open set intersects $U$ in a nonempty open set. By the 
lemma, that intersection lies outside the image of $F$, 
contradicting (1). 
\end{proof}

\begin{Cor}\label{odd:dense}
For $F$ as above, if the field extension is of odd degree, 
then the image of $F$ is dense. 
\end{Cor}

\begin{proof}
A polynomial of odd degree has a real root. 
\end{proof}

\begin{Ex}
For $y=f(x)=x^2$ and $x$ chosen for $t$, $R(T)=T^2-y$, which is irreducible over 
$\R(Y)$, but factors over $\R(X)$ as $(T-x)(T+x)$. Note 
that its specializations $R(y)(T)$ for $y < 0$ do not factor 
over $\R$.
\end{Ex}

\subsection{An injectivity criterion}\label{conj4A}{\ \newline}

Start this section with the same notations and assumptions 
as in the preceding one. In particular, $j(F)$ may have 
zeros, but does not vanish identically on $\R^n$. 
However, $t$ will no longer be a completely 
arbitrary primitive element, but instead will be selected using 
the following well known results. 

\begin{Lemma}
Some real linear combination
$c_1x_1 + c_2x_2 + \cdots + c_nx_n, c_i \in \R$ 
of the coordinate polynomials is primitive. 
\end{Lemma}

\begin{proof}
Choose $c_1=1$. 
For some $c_2 \in \R$, the subfield of 
$\R(X)$ generated over $\R(F)$ by the single element 
$x_1 + c_2x_2$ is the same as the subfield generated by 
the two elements $x_1$ and $x_2$. 
For otherwise, since there are only finitely many intermediate 
fields in characteristic zero, there would be combinations 
$x_1+ c_2x_2$ with different values of $c_2$ that lie in 
the same proper subfield, which must therefore contain 
$x_2$ and hence also $x_1$, a contradiction. 
The ultimate result follows by adding one summand at a time, 
to obtain a linear combination that generates 
$\R(X)=\R(\vars)$, and so is primitive by definition. 
\end{proof}

Remark. Almost all linear combinations are primitive. 
More specifically, the coefficients 
$(c_1,c_2,\ldots,c_n) \in \R^n$ of those that are not 
primitive belong to a finite union of proper vector subspaces 
of $\R^n$, one subspace for each proper subfield of 
$\R(X)$ containing $\R(F)$. 

Recall the identification of $y_i$ with $f_i$, and hence of 
$\R[Y]$ and $\R[F]$.

\begin{Lemma}
If $t$ is a primitive element, there is a multiple 
$at$ of it, for a polynomial $a \in \R[Y]$, such that 
$at$ is primitive and its  monic minimal polynomial 
has coefficients in $\R[Y]$. 
\end{Lemma}

\begin{proof}
Let $R(T) \in \R[Y][T]$ be the minimal degree polynomial
with root $t$ of the previous section, 
and $a \in \R[Y]$ the 
coefficient of its leading term as a polynomial of degree $d$ in 
$T$. From $R(t)=0$, it follows that $at^d$ is a sum of 
terms of lower degree in $t$. But then $(at)^d=a^{d-1}(at^d)$ can 
be written as a sum of terms of lower degree in $(at)$. 
That yields a degree $d$ monic polynomial 
$S(T) \in \R[Y][T]$, such that $S(at)=0$. 
Since $0 \ne a \in \R(Y)$, the field generated by 
$at$ over $\R(Y)$ is also $\R(X)$. So $S$ is the monic 
minimal polynomial of $at$. 
\end{proof}

\begin{Thm}\label{better}
Let $F:\R^n \To \R^n$ be an everywhere defined real rational 
map with a \J determinant that is not identically zero
on $\R^n$. 
Then there exists a primitive element $t$ for the associated 
function field extension $\R(F) \subseteq \R(X)$, such that
\begin{enumerate}
\item$t$ is defined everywhere on $\R^n$, and 
\item $t$ is integral over $\R[F]$. 
\end{enumerate}
\end{Thm}

\begin{proof}
Use the two lemmas in succession. 
After the first, one has a polynomial primitive element.
If $t$ is a polynomial  ($t \in \R[X]$), the same is 
not necessarily true of $at$, since $a$ is a polynomial in 
the components of $F$, which are not assumed to be polynomial. 
But $at$ will be an everywhere 
defined real rational function. 
Rename it $t$. 
\end{proof}

Remark.  From the proof, $t$ can be chosen more specifically 
as the product of a linear form in the $x_i$ 
and a polynomial in the $f_i$. 

Assume in the following that $t$ satisfies (1) and (2). 
$R(T)$ will denote some 
nonzero multiple, by a real constant, of the monic 
minimal polynomial of $t$. 
Condition (1) ensures that for all $x \in \R^n$ there is 
a corresponding real root $t(x)$ of $R(y)(T)$ at $y=F(x)$. 
Condition (2) implies that for all $y \in \R^n$ the roots of $R(y)(T)$ are 
continuous functions of $y$. 
One way of stating continuity of roots more  
precisely is that there exist $d$ continuous 
functions $r_i:\R^n \To \C$,
such that $R(y)(r_i(y))=0$  for any $y \in \R^n$, 
and all roots are obtained in this way. 
A canonical way to define the $r_i$ is to totally order 
$\C$ using a lexicographic ordering of the real and complex 
parts, and then let $r_i(y)$ be the $i^{th}$ element of the 
set of all $d$ roots of $R(y)(T)$, where the roots are 
repeated according to multiplicity and arranged in order 
from smallest to largest \cite{RootsDepend}. 
The set of all distinct roots for all $y \in \R^n$ is then the 
closed subset of $\R^n \times \C$ that is the 
union of the graphs of the continuous functions $r_i$; 
the individual function graphs will also be called sheets. 
Each sheet is closed in $\R^n \times \C$ and its projection to 
$\R^n$ is a homeomorphism. 
The projection of the set of roots onto $\R^n$ is 
clearly proper. 
A consequence is that roots over points sufficiently close to
a given point $y$ are each near a uniquely determined 
distinct root over $y$. 
It is obvious that repeated roots are those that lie on 
more than one sheet, and it follows from the above 
consequence of properness
that the multiplicity of a repeated root is exactly the number 
of sheets on which it lies. 

\begin{Ex}
Let $R(y)(T)=T^2-(y^2+y^4)T+y^6=(T-y^2)(T-y^4)$. 
Since all the roots are real, they are ordered in the usual 
way if the above method is used, but the slightly unnatural 
sheets $r_1=\min(y^2,y^4)$ and $r_2=\max(y^2,y^4)$ are 
produced, rather than the natural algebraic sheets 
$y^2$ and $Y^4$. 
\end{Ex}

Denote by $\#r$ ($\#c$) the number of real (respectively, 
complex) roots counted with multiplicities. 
Here, a complex root is understood to be a root with a 
nonzero complex part. 
For any integer $i$, the condition $\#c \ge i$ 
(equivalently, $\#c > i-1$) is open, 
meaning that the set of points $y \in \R^n$ at which it holds 
is an open set, simply because $\R$ is a closed subset of 
$\C$. And $\#r \ge i$ (equivalently, $\#r > i-1$) is closed, 
as it is the negation of $\#c > d-i$. 
Also, any real root with arbitrarily close complex roots is 
repeated, since complex conjugate roots lie on different 
sheets.

One can generalize Theorem \ref{Dense} as follows. 
Let $E_1$ be the closure of the image of $F$. 
Given that $j(F)$ does not vanish 
identically, the image of $F$ contains some interior points, 
and hence so does $E_1$. 
Let $O$ be the complement of $E_1$ in $\R^n$, and 
$E_2$ the closure of $O$. $E_2$ is empty if, and only if, 
$F$ has dense image, and otherwise contains some 
interior points. 
If $F$ does not have dense image, then $E_1$, $E_2$, and 
$E_1 \cap E_2$ are all closed and nonempty, since $\R^n$ is 
connected.

\begin{Prop}\label{nonDense}
If $y \in E_1$, then $R(y)(T)$ has at least one real root. 
If $y \in E_2$, then every real root of $R(y)(T)$ has 
multiplicity greater than one. 
\end{Prop}

\begin{proof}
If $y=F(x)$, then $t(x)$ is a real root over $y$. 
Since $\#r \ge 1$ is a closed condition, it holds on 
$E_1$. By Lemma \ref{roots} of section \ref{conj4}, 
there is a Zariski open subset $U \cap O$ of $O$ 
consisting of points over which $R(T)$ has only 
complex roots. So every real root over a point of 
$E_2$ has arbitrarily close complex roots. 
\end{proof}

The notion of sheets can be used to prove a significant 
criterion for injectivity. Some prerequisite terminology follows. 
A map of topological spaces is said to be locally injective at a 
point if it is injective on some neighborhood of that point. 
By the celebrated Invariance of Domain Theorem of Brouwer, 
a continuous injective map of an open subset of $\R^n$ to 
$\R^n$ has an image that is an open subset of $\R^n$, and is 
a homeomorphism onto that image. Since $F$ is 
continuous, then if it is locally injective at a point, it is also an 
open map at that point. 

\begin{Lemma}\label{twos} 
Suppose $a \ne b$ are points in $\R^n$, 
and $F$ is locally injective at both $a$ and $b$.  
If $F(a)=F(b)=y$ and $t(a)=t(b)=r$, 
then $r$ is a repeated root of $R(y)(T)$.
\end{Lemma}

\begin{proof}
Suppose, to the contrary, that $r$ is a simple root. 
Then there is an open subset $O$ of the set of roots (or of $\R^n \times \C$, for that matter) that contains $(y,r)$ and  no points on any other sheet than 
the sheet on which $(y,r)$ lies. 
The inverse image of $O$ under the map $(F(x),t(x))$ is 
an open subset of $\R^n$ that contains both $a$ and $b$. 
By the local injectivity hypothesis, it contains two disjoint 
open sets $U_a$ and $U_b$, containing
 $a$ and $b$, respectively, 
such that $F$ is injective, hence open,  on each of 
$U_a$ and $U_b$. 
$F(U_a) \cap F(U_b)$ is therefore an open neighborhood of 
$y$. By Lemma \ref{roots}, it contains a point $y'$, such 
that real roots over $y'$ correspond bijectively to inverse 
images of $y'$ under $F$. 
Take an inverse image of $y'$ in $U_a$ and one, 
necessarily different, in $U_b$. 
The corresponding roots are then distinct. 
That contradicts the fact that their images under 
$(F(x),t(x))$ both lie in $O$, and hence on a 
single sheet. 
\end{proof}

Remark. Local injectivity is used in the proof only to deduce 
that $F$ is open at $a$ and $b$. 
Still, it seems the appropriate hypothesis in attempts to 
prove injectivity on a larger scale.

Note that if $A \subseteq \R^n$, then, 
unless $A$ is open, the 
assertion that $F$ is locally injective at every point $a \in A$ is 
not a statement about the values of $F$ on $A$, but rather 
about its values on open neighborhoods of the points of 
$A$ in $\R^n$. 
The terminology for roots and sheets will be slightly simplified. 
At a point $x \in \R^n$, $t$ will be said to determine a simple  
(repeated) root if $(F(x),t(x))$ lies on only one 
(respectively, more than one) sheet, without any explicit 
reference to the polynomial $R(F(x)(T)$ of which $t(x)$ is 
a root.

\begin{Thm}\label{InjCrit} (Injectivity Criterion) 
Let $F:\R^n \To \R^n$ be an everywhere defined real rational 
map with a \J determinant that is not identically zero
on $R^n$. 
Let $t$ be an everywhere defined  primitive element for the associated function field extension $\R(F) \subseteq \R(X)$, such that $t$ is integral over $\R[F]$. 
Suppose that $A \subseteq \R^n$ is connected, that $F$ is 
locally injective on (some open neighborhood of) $A$, and 
that $t$ determines only simple roots on $A$. Then 
$F$ is injective on $A$. 
\end{Thm}

\begin{proof}
Let $s$ be the restriction to $A$ of the 
continuous map $(F(x),t(x))$ to the set of roots. 
The inverse images of the sheets are closed subsets of $A$. 
If two of them intersect at a point $a$, then $s(a)$ lies on 
more than one sheet, and so $t(a)$ is a repeated root. 
But that contradicts the hypotheses, so there can be no such 
point of intersection. Since $A$ is connected, only one of 
the disjoint closed subsets is nonempty. 
Hence $s(A)$ is contained in a single sheet. 
Now suppose that $a,b \in A$, with $a \ne b$, but 
$F(a)=F(b)=y$. The points $(y,t(a))$ and $(y,t(b))$ lie 
on the same sheet and have the same projection to $\R^n$, and 
so are the same point. That is, $t(a)=t(b)=r$. 
But then $r$ is a repeated root by Lemma \ref{twos}, 
a contradiction. 
So $F(a)=F(b)$ is not possible. 
\end{proof}

\subsection{Applications}\label{conj4B}{\ \newline}

This section concerns some applications of 
previous results, 
primarily the injectivity 
criterion of Theorem \ref{InjCrit}, 
in the RRJC context. 
That is, the notations and assumptions of the preceding section 
apply, but it is also now assumed that $j(F)$ vanishes nowhere.

Briefly, $F \in \R(X)^n$ and $t \in \R(X)$ are defined everywhere, 
$j(F)$ vanishes nowhere, $\R(F)(t)=\R(X)$, and $t$ is integral over 
$\R[F]$, with $R(T) \in \R[Y][T]$ a degree $d$ nonzero 
real constant multiple of the monic minimal polynomial 
of $t$ over $\R[Y]=\R[F]$.

\begin{Lemma}\label{simple}
If $t$ determines only simple roots on a connected set $A$, 
then $F$ is injective on $A$. 
\end{Lemma}

\begin{proof}
Follows immediately from the injectivity criterion, 
since $F$ is locally bijective everywhere. 
\end{proof}

Let $R'(T) \in \R[Y][T]$ be the 
partial derivative of $R(T)$ with respect to $T$. 
A root $r$ of $R(y)(T)$ is a repeated root if, and only if, 
it is also a root of $R'(y)(T)$. Suppose that $t(x)=r$. 
Differentiate the equation $R(F(x))(t(x))=0$ 
with respect to $\vars$, obtaining 
\begin{equation}\label{nabla}
(\nabla_R\cdot M + R'(r)v = 0,
\end{equation}
where $\nabla_YR$ is the row $n$-vector 
of partials of $R(T)$ with respect to the $Y$ coordinates, 
$M$ is the \J matrix of $F$ at $x$, 
$\cdot$ indicates a vector matrix product,
and $v$ is the row $n$-vector of partials of $t$ evaluated 
at $x$. 
The components of $\nabla_YR$ belong to $\R[Y][T]$, 
and in equation \ref{nabla} they 
must, of course, be evaluated not only at 
$y=F(x)$, but also at $T=r=t(x)$. 
Write $\nabla_YR(x)$ for $\nabla_YR(F(x))(t(x))$. 
Then $\nabla_YR$ is an everywhere 
defined real rational vector field on the domain of $F$. 

The case $d=1$ is exceptional, since it represents the only 
situation in which $t$ can be constant. 
More on this later. For the moment, just assume that $t$ is 
not constant.

\begin{Lemma}\label{vf}
$\nabla_YR$ is not the zero vector field, and if 
$\nabla_YR(x) \ne 0$, then $t(x)$ is a simple root 
of $R(y)(T)$ at $y=F(x)$. 
\end{Lemma}

\begin{proof}
By construction, $R(T)$ has only simple roots 
over the Zariski open set $U$ of Lemma \ref{roots}. 
$F^{-1}(U)$ is a nonempty open set, and it must contain 
a point at which the gradient vector of $t(x)$ is nonzero, 
since $t(x)$ is a nonconstant real rational function, 
and so is not locally constant anywhere. 
Evaluating equation \ref{nabla} at that $x$, the scalar 
$R'(r)$ is nonzero, because $r=t(x)$ is a simple
root, and the gradient vector $v$ is also nonzero. But that is 
clearly impossible if $\nabla_YR(x)=0$, and so 
$\nabla_YR$ cannot be the zero vector field. 
For the second conclusion, since $M$ is nonsingular 
at any $x \in \R^n$, if $\nabla_YR(x) \ne 0$, then 
$R'(r)$ cannot be zero, and hence $r=t(x)$ is a 
simple root. 
\end{proof}

\begin{Prop}\label{1:1}
Suppose that $\nabla_YR$ has no zeros on a connected set $A$. 
Then $F$ is injective on $A$. 
\end{Prop}

\begin{proof}
$t$ determines only simple roots 
on $A$, by Lemma \ref{vf}. 
Now use Lemma \ref{simple}. 
\end{proof}

Since $\nabla_YR$ is an everywhere defined real rational 
vector field, its set of zeros is a real algebraic set. 
Points, at which all the components of a vevtor field 
vanish, are usually called singular points of the vector 
field. So the set of zeros of $\nabla_YR$ will be 
denoted by $S$. 
Note that $S$ is of codimension at least $1$ in $\R^n$, 
since $\nabla_YR$ is not the zero vector field.

\begin{Cor}\label{codimS}
If $S$ has codimension $2$ or more ( in particular, 
if $S$ is empty), them $F$ is globally injective.
\end{Cor}

\begin{proof}
By reason of the codimension assumption, 
$A=\R^n \setminus S$ is connected (see section \ref{conj3}), 
so $F$ is injective on $A$. Also, $A$ is Zariski open, 
so $F$ is injective on $\R^n$ by Lemma \ref{injective}. 
\end{proof}

Remark. Of course, $F$ is then invertible. This situation 
is similar to that in section \ref{conj3}, where the 
codimension of the asymptotic variety was considered.
Chances seem better here, since the naive expectation 
for the dimension of the singular points of a vector 
field is zero.

Perhaps more practically, one always has injectivity 
on each connected component of $\R^n \setminus S$.

Back to the case $d=1$. This is the birational case, so 
$F$ is globally injective. Since $\R[Y]$ is integrally closed in 
$\R(Y)=\R(X)$, $t$ must be a polynomial in $\R[Y]$. 
If that polynomial is a constant, then $\nabla_YR$ is 
identically zero. If not, then $S$ is the inverse image 
under $F$ of the set of zeros of 
the gradient vector field of $t \in \R[Y]$ on the 
codomain.

\begin{Ex}\label{Sex}
This example works out the details for the specific
Pinchuk map defined in  section \ref{specific}. 
$F=(P(x,y),Q(x,y))$ and the auxiliary polynomials $t,h,f,q,u$ 
have their meanings here as there. 
The selected primitive element is $h$, and section \ref{ext} 
shows that 
\begin{equation*}
R(h) = 
(197/4) h^6 + \cdots +(2PQ - 170P^3)h -P^2Q = 0.
\end{equation*}
for a polynomial $R(T) \in \R[P,Q][T]$, which is 
fully written out in the appendix (section \ref{sup}, 
equation \ref{fullR}). 
The companion vector field is 
$(\partial R/\partial P, \partial R/\partial Q)$,
evaluated at $T=h$. 

$R$ has only three terms involving $Q$ making 
it easy to 
compute 
$\partial R/\partial Q = -(T-P)^2$. 
The expression for $\partial R/\partial P$ is 
considerably more complicated, but on substituting 
$T$ for $P$, it simplifies to $ T^3(6T^2+14T+8)$,
which is conveniently independent of $Q$. 
At any singular point of the vector field, therefore, both 
$P=h$ and $ h^3(6h^2+14h+8)=0$, must hold. 
And, conversely, any such point 
belongs to $S$, the set of all singular points 
of the vector field.

$S$ can be determined more specifically from the 
equations 
$t=xy-1,h=t(xt+1),f=(xt+1)^2(t^2+y),P=f+h$ that 
define $P$. 
And $F(S)$ can be determined by evaluating $Q=q-u$,
 where $q=-t^2-6th(h+1)$ and $u=u(f,h)$ is defined by 
equation \ref{ueq} in section \ref{specific}. 
If $P=h$, then $f=(xt+)^2(t^2+y)=0$. 
The case $xt+1=x^2y-x+1=0$ is equivalent to 
$x \ne 0$ and $y=1/x-1/x^2$ and it is easy to show that
$P=h=0$ and $Q=-1/x^2$ on the two curves.
The case $t^2+y=x^2y^2-2xy+y+1=0$ is equivalent to
$y<0$ and $x=1/y \pm 1/\sqrt{-y}$, and on these two 
curves $P=h=-1$ and $Q=y-u(0,-1)=y-163/4$. 
Since both $0$ and $-1$ are roots of $h^3(6h^2+14h+8)$, 
$S$ is the union of these four curves. They are disjoint, 
since the curves in each pair are disjoint and the image 
under $F$ of each curve in the first case is the negative 
$Q$-axis, while in the second case it is the portion of 
the line $P=-1$ satisfying $-\infty <Q<-163/4$. 
All four branches are asymptotic to the $y$-axis at 
$y=-\infty$ and to the $x$-axis at either $x=-\infty$ 
(three times), or at $x=+\infty$ (case 1, $x>0$). 
Also, $y$ is bounded above on all branches, with a maximum 
value $y=1/4$ at $x=2$ (case 1).

$S^c=\R^2 \setminus S$ is the disjoint union 
of five unbounded connected open 
sets. Each region has a boundary consisting of one or two 
branches of $S$. 
By Proposition \ref{1:1}, $F$ is injective on each of 
those five regions. 
Recall that $B(F)=F^{-1}(A(F))$ consists of three 
curves, and its complement of four regions, with 
$F$ injective on each region. 
Although the closure of $F(S)$ intersects $A(F)$ (at 
$(-1,-163/4)$ and $(0,0)$),  $F(S) \cap A(F) = \emptyset$. 
So each of the 
curves composing $B(F)$ lies in $S^c$ and, since it 
is connected, in just one of its five component regions. 
In fact, the component curve of $B(F)$ on which 
$-1 <P<0$ lies in the region of $S^c$ on which $y$ is 
not bounded above (the 'top' region), whereas the other two 
curves lie in the adjacent region bounded by the two branches
of $S$ $x=1/y+1/\sqrt{-y}$ ($y<0$) and 
$y=1/x-1/x^2$ ($x<0$). 
 These assertions can be checked by adding a few more 
curves with known images to $B(F)$, so as to form 
connected configurations of curves, with images that do  
not intersect $F(S)$, and then determining 
the location of a single point in each configuration 
 relative to the branches of $S$. 
For details, see the appendix.

Dually, each branch of $S$ is contained in one of the four 
regions of $\R^2 \setminus B(F)$. 
Two of the regions each map diffeomorphically onto the 
connected component of $\R^2 \setminus A(F)$ containing 
the positive $P$-axis 
(see the figure in section \ref{specific}). Since 
their image region contains no point of $F(S)$, they 
cannot contain any branch of $S$.
Both the case 1, $x<0$ and the case 2, $x=1/y-1/\sqrt{-y}$ 
branches lie in the same region, because there is no 
component curve of $B(F)$ to separate them in the region of 
$S^c$ that they bound.. 
Finally, the other two branches of $S$ must lie in the 
remaining, fourth, region, because they have the same 
images as the two in the third region and $F$ is 
injective on each region. 
\end{Ex}

\subsection{The {G}alois case}\label{conj5}{\ \newline}

As before, let $F:\R^n \To \R^n$ be an everywhere defined real 
rational map with nowhere vanishing \J determinant, $A(F)$ 
the set of points at which $F$ is not proper, and 
$B(F) = F^{-1}(A(F))$. 
Let G be the group (under composition) of real birational maps 
$g:\R^n \To \R^n$, such that the corresponding automorphism 
$g*$ of $\R(X)$ belongs to the group $G*$ of 
automorphisms of $\R(X)/\R(F)$, that is, such that $g*$ 
preserves every element of $\R(F)$. 
$G$ and $G*$ are opposite groups; that is, abstractly the same 
except for a reversal of the order of the product operation. 
By construction, every $g \in G$ satisfies $F \circ g = F$. 

\begin{Lemma}\label{rigid}
Any $g \in G$ is completely determined by its value at any one 
point at which it is defined. 
\end{Lemma}

\begin{proof}
Let $a \in \R^n$ be a point at which $g$ is defined, and 
let $b=g(a) \in \R^n$ be its value there. Let $U_a,U_b$ be 
open sets containing $a$ and $b$, respectively, that are mapped 
homeomorphically by $F$, with $g$ defined on $U_a$. 
Since $g$ satisfies $F(g(x))=F(x)$ at points where it is 
defined, $g$ is completely determined on the open set 
$U_a \cap g^{-1}(U_b)$. The components of $g$ are rational 
functions on $\R^n$, each determined by its restriction as a rational 
function to any open subset of $\R^n$. 
\end{proof}

Let $W$ be the complement of $B(F)$ in the domain of $F$. 
Recall, from section \ref{conj}, that $W=\R^n \setminus B(F)$ 
is an open semi-algebraic subset of $\R^n$, that the  same is true  for each of the finitely many connected components $V$ of $W$, and that, 
moreover, each such $V$ is a connected cover of finite degree 
of its image $U=F(V)$ via the map induced by $F$. 

\begin{Lemma}\label{W}
Each $g \in G$ is defined at every point of $W$, and maps $W$ 
homeomorphically onto $W$. 
\end{Lemma}

\begin{proof}
The set of points where $g$ is defined is Zariski open, 
hence it intersects each $V$. Suppose $g$ is defined at $a \in V$, 
but not at $b \in V$. Let $c(t)$ be a continuous curve 
$[0,1] \To V$ wit $c(0)=a$ and $c(1)=b$. Replacing $b$, if 
necessary, by the first point on the curve after $a$ at which $g$ is 
not defined, one can assume that $b$ is a boundary point of the 
points of definition of one or more of the components of $g$, 
with $g$ defined at all the other points on the curve. 
Lift the curve $F(c(t))$ in $U=F(V)$ to a curve starting at 
$a'=g(a)$ in whatever total space $V'$ contains $a'$. 
There is no guarantee either that $V$ and $V'$ are the same 
or that they are different, only that they share the same base 
space $U$. Let $b'\in V'$ be the endpoint of the lifted curve. 
Using two open sets, one containing $b$ and the other $b'$, 
both mapped bianalytically by $F$ onto the same open subset 
of $U$, and employing a slight variation of the argument in 
Lemma \ref{rigid}, it is clear that $g$, and thus each of its 
components, can be analytically extended to $b$. 
By Lemma \ref{Analytic} in section \ref{conj1}, that is a 
contradiction. Therefore, there is no point $b \in V$ at 
which $g$ is not defined, and since $V$ was any connected 
component of $W$, it follows that $g$ is defined on $W$. 
If $a \in W$, then $g(a)\in W$, because $F(g(a))=F(a)$ 
does not belong to $A(F)$. The same considerations apply 
to the inverse element of $g$ in the group $G$, so $g$ is a 
homeomorphism of $W$ onto $W$. 
\end{proof}

\begin{Prop}\label{Gacts}
$G$ acts freely on $W$ as a finite transformation group. 
In particular, no element of $G$, except the identity, has a 
fixed point, and the size of the orbit of any point is the 
number of elements (the order) of $G$. 
\end{Prop}

\begin{proof}
Combine the two preceding lemmas.  
\end{proof}

\begin{Cor}
The map of $W$ onto $F(W)=F(\R^n)\setminus A(F)$ is 
exactly $d$-to-$1$. 
\end{Cor}

\begin{proof}
Clear.
\end{proof}

\begin{Prop}\label{G=1}
If $F$ has an inverse, then the identity is the only automorphism 
of $\R(X)$ that preserves every element of $\R(F)$. 
\end{Prop}

\begin{proof}
$F$ maps all the points of an orbit to the same point, so if 
$F$ is injective then $G$ must be trivial. 
\end{proof}

That is a necessary condition on the extension for the existence 
of an inverse for $F$. By Theorem \ref{extension} in section 
\ref{main}, any {P}inchuk map shows that it is not sufficient.

Call $F$ Galois over $\R$, or just Galois for short, if the extension $\R(X)/\R(F)$ is Galois. If $F$ is defined over 
a subfield $k \subset \R$, similarly define 'Galois over $k$' for  
$F$ and note that it implies that $F$ is Galois over $\R$, 
and that the extensions are of the same degree with canonically isomorphic Galois groups.  If $F$ is Galois, then $G$ is the opposite group of the Galois group. 

\begin{Thm}\label{galcase} ({G}alois case of the RRJC) 
Let $F:\R^n \To \R^n$ be a real rational everywhere defined map 
with nowhere vanishing \J determinant. Suppose that $F$ is 
{G}alois. Then the following are equivalent:
\begin{enumerate}
\item $F$ has a global real analytic inverse, 
\item $F$ has an everywhere defined rational inverse,
\item $F$ is birational, 
\item the {G}alois group is trivial.
\end{enumerate}
\end{Thm}

\begin{proof}
Use Proposition \ref{G=1} and Theorem \ref{bicase} (the 
birational case) to prove implications in the order 1,4,3,2,1.
\end{proof}

That is not the hoped for result. A full analogue of the known 
{G}alois case ( polynomial maps with \nz constant \J determinant) would be that the equivalent conditions in the above theorem must be true. 
In other words, that any {G}alois extension in this 
situation is of degree one.

There are some special results in the {G}alois case, 
reported here without proof, 
$B(F)$ is algebraic, not just closed semi-algebraic. 
If $y$ lies in the closure 
of the image of $F$, 
then $y \in A(F)$ if, and only if, it has fewer than 
$d$ inverse images. 
Let $t$ and $R(T)$ be as in the two immediately 
preceding sections. 
Under the same conditions on $Y$, (a) all the roots of 
$R(y)(T)$ are real, and 
(b) if all $d$ roots are distinct, they are the values of 
$t$ at $d$ distinct inverse images of $y$. 
Note that if $F$ has dense image, then any 
$y \in \R^n$ qualifies, and note that only distinct roots 
are needed in (b), with no further requirement 
that $y$ be generic. 
Unfortunately, these special properties shed no light on 
the question of whether the extension must be 
birational.

\begin{Ex}
Let $F:\R^n \To \R^n$ be the map with components the $n$ 
elementary symmetric functions in the variables $\vars$. 
The extension is {G}alois, but the \J condition is not met. 
This example is useful for geometric 
visualization of the group action. 
\end{Ex}

\subsection{Modified conjectures}\label{conj5A}{\ \newline}

Let $F:\R^n \To \R^n$ be an everywhere defined real rational 
map, with a nowhere vanishing \J determinant.  
As in section \ref{conj}, the introductory section on real 
\J conjectures, let $d$ be the degree of the associated 
function field extension 
$\R(X)/\R(F)$, and let $0<N \le d$ be the 
maximum number of inverse images under $F$ of any point in 
the codomain of $F$.

\begin{Lemma}
$d-N$ is even.
\end{Lemma}

\begin{proof}
The set of points in the codomain $\R^n$ with $N$ inverse 
images is open (section \ref{conj}). 
Let $t \in \R(X)$ be a primitive element for the extension, and 
$m(T) \in \R(f)[T]$ its minimal polynomial over $\R(F)$. 
By Lemma \ref{roots} in section \ref{conj4},  
$m(T)$ is defined and has distinct roots, 
with exactly $N$ of them real, 
over a nonempty  open subset of the codomain 
$\R^n$. It suffices to note that at a point of that subset, 
$d-N$ is the number of complex roots, which occur in complex conjugate pairs. 
\end{proof}

\begin{Cor}
If $F$ is invertible, then $d$ is odd.
\end{Cor}

\begin{proof}
$N=1$.
\end{proof}

That suggests the following 

\begin{Conj*}\label{MRRJC} 
((MRRJC) Let $F:\R^n \To \R^n$ be an everywhere 
defined real rational 
map, whose \J determinant vanishes nowhere on $\R^n$. 
If the associated function field extension is of odd degree, 
then $F$ has a global real analytic inverse. 
\end{Conj*}

The acronym MRRJC stands for  modified 
rational real \J conjecture. 
This conjecture is not vacuous, as is shown by the examples 
$y=f(x)=x+x^d$ for $d > 1$ odd. The condition that 
$d$ is odd can be replaced by the geometrically more 
natural condition that $N$ is odd. 
These conditions are equivalent and necessary. 
In the RRJC context, 
the grandiose conclusion is equivalent to the simple 
statement that $F$ is injective, or that $N=1$. 
Indeed, injectivity implies not only the existence of 
an inverse, but also that the inverse is both real analytic 
and semi-algebraic, hence a Nash diffeomorphism. 
In sum, the conjecture is that $N$ odd implies $N=1$. 
One piece of evidence in favor of the conjecture is that 
its hypotheses imply, by Theorem \ref{Dense} and 
Corollary \ref{odd:dense}, that the image of $F$ is dense 
in $\R^n$ and its complement, $\R^n \setminus F(\R^)$, 
is contained in a real algebraic strict subset of $\R^n$.

Remark. There is another necessary condition for 
invertibility that applies to the function field extension. 
Namely, by Proposition \ref{G=1} in section \ref{conj5}, 
the automorphism 
group of the extension must be trivial. It has not been included as an additional hypothesis in the MRRJC by deliberate choice, 
partly because it does not, by itself, even exclude 
the Pinchuk counterexamples, whose extensions have 
trivial automorphism groups by Theorem \ref{extension} 
in section \ref{main}.

Turn next to a definition of $N$, suitable for this 
section, in a somewhat more general situation. 
It is a familiar fact that if either the rationality or the 
\J condition is dropped, there may be no finite upper  bound on 
the number of inverse images. Traditional examples are 
$F=(e^xcos(y),e^xsin(y))$ and $F=(x,xy)$. 
For such maps, assign $N$ the symbolic 
value $\infty$, regardless of specifics of the cardinalities 
of various fibers. 
This may occur even if the map is quasifinite, meaning 
each individual fiber is a finite set. 
If there is a finite bound, let $N$ be the least such bound. 
In this way, $N$ is defined for any map 
$F:\R^n \To \R^n$ whatsoever, and if it is finite, then it is 
the (finite) maximum of the (finite) cardinalities of all the 
fibers of $F$.

Two maps, $F$ and $G$, from a topological space $A$ to 
another one $B$, are called topologically equivalent 
if $F = h_B \circ G \circ h_A$, where $h_A$ and $h_B$ are 
homeomorphisms, respectively of 
$A$ to itself and of $B$ to itself. 
In other words, $F$ and $G$ are the same map up to 
coordinate changes in the domain and codomain by 
topological automorphisms. 
Topological stable equivalence for the set of all maps 
$F:\R^n \To \R^n$ in all dimensions $n>0$ is the 
equivalence relation generated by (1) topological 
equivalences, and (2) the equivalence of any map 
$F=(\comps)$, and its extension 
$(f_1,\ldots,f_n,x_{n+1},\ldots,x_m)$
to any larger dimension $m$. 
There are many other types of stable equivalence, such as 
real analytic or polynomial, each characterized by the type 
of automorphisms allowed for (global) coordinate changes. 
Stable equivalence, unqualified, will refer to the 
least restrictive, purely set theoretic, type, 
with all bijections allowed as automorphisms. 
Clearly
stable equivalence preserves $N$, as defined above. 
That is, two stably equivalent maps have the same value 
of $N$, whether finite or $\infty$. 
Of course, this stable equivalence does not preserve 
rationality or even the existence of a \J matrix. 
Nonetheless, if $F$ and $G$ are stably equivalent and 
both are everywhere defined real rational maps with nowhere 
vanishing \J determinant, then they are equivalent as 
far as the conjecture is concerned, since both $N$ odd 
and $N=1$ are preserved.

For brevity, call $F:\R^n \To \R^n$ (1) nondegenerate if 
$j(F)$ is not identically zero, 
(2) nonsingular if $j(F)\ne 0$ everywhere, and (3) a Keller 
map if $j(F)$ is a nonzero constant. 
These terms are meant to imply that $J(F)$, the  \J matrix 
of $F$, exists at 
every point of $\R^n$, and can be applied to any such 
$F$ if the corresponding restriction on 
$j(F)$ is satisfied.
For polynomial stable equivalence, the applicable 
automorphisms are \p maps with \p inverses, making 
it obvious that such equivalence preserves (in both 
directions) each of the above three properties.

There are two classic reductions of the ordinary JC 
to Yagzhev maps \cite{Jagzhev,BCW82} and to \Druz maps \cite{Effective}.
A Yagzhev map is a polynomial map of the form 
$F=X+H$, where $X=(\vars)$, and each 
component of $H$ is a cubic homogeneous polynomial in
the variables $\vars$. 
Yagzhev maps are also called maps of cubic homogeneous type. 
A \Druz map (or map of cubic linear type)
is a Yagzhev map, 
for which the components of $H$ are cubes of linear 
forms ($h_i=l_i^3$). In a departure from the convention in 
some other works, 
these definitions impose no restriction on $j(F)$, 
beyond the obvious $j(F)(0)=1$. 
Note, however, that a Yagzhev map $F=X+H$ is a Keller map 
if, and only if, $J(H)$ is nilpotent, since both assertions are 
just different ways of saying that the formal power series 
matrix inverse of $J(F)$ is polynomial.

Reduction theorem proofs use the strategy of 
transforming an original map into a map 
of the desired form in a succession of steps that preserve 
the truth value of certain key properties (and typically 
increase the number of variables). 

For the JC, $\C$ is usually selected as the ground field and 
the key properties are the Keller property 
and the existence of a polynomial inverse.
But the strategy and specific steps can be applied more 
generally than just to polynomial Keller maps and yields, 
for instance, a reduction of the SRJC to the cubic linear 
case. \cite{Effective}.

Historical Note. 
At the 1997 conference in Lincoln, Nebraska, to honor 
the mathematical work of Gary H. Meisters. it was suggested 
by T. Parthasarathy that the SRJC reduction be attempted 
for the 1994 counterexample of Pinchuk. 
The challenge was taken up by Engelbert Hubbers, and 
in 1999 he demonstrated the existence of a 
counterexample to the SRJC of cubic linear type, 
coincidentally in dimension 1999. 
He started with exactly the specific Pinchuk map of 
section \ref{specific}, used a computer algebra system
to verify a human guided reduction path to a Yagzhev map
in dimension 203, then explicitly computed a
Gorni-Zampieri pairing \cite{GZpairs} 
to a \Druz map in dimension 1999, using sparse matrix 
representations as necessary.
These details are excerpted from a comprehensive 
unpublished note by Hubbers, which he made available.

If the MRRJC is considered only for polynomial maps, it 
becomes the MSRJC, a modified strong real \J conjecture. 
Because of the long standing and continuing interest in 
the SRJC, a separate full statement is warranted.

\begin{Conj*}\label{MSRJC} 
(MSRJC) Let $F:\R^n \To \R^n$ be a real polynomial
map, whose \J determinant vanishes nowhere on $\R^n$. 
Suppose that the associated function field extension 
is of odd degree, 
or, equivalently, that the maximum cardinality of 
the fibers of $F$ is odd. 
Then $F$ has a global real analytic inverse. 
\end{Conj*}

There are reduction theorems for the M	SRJC parallel to those 
just discussed for the JC. 
To reduce the MSRJC to the cubic homogeneous case, 
it suffices to take 
reduction steps that preserve $N$ and to transform only 
nonsingular maps, in which case (2) below can be stated more 
simply as the equality of the maximum cardinality of the 
fibers of $F$ and $G$. 
The proof of the following theorem 
basically follows \cite{Effective}, with simplifications 
suggested by Michiel de Bondt.

\begin{Thm}\label{reduce-1}
There is an algorithm that transforms a nondegenerate, 
polynomial map 
$F:\R^n \To \R^n$ into a map $G:\R^m \To \R^m$ of cubic homogeneous type, where $m$ is generally much larger 
than $m$, such that
\begin{enumerate}
\item $F$ is nonsingular if, and only if, $G$ 
is nonsingular, and
\item a finite bound on the cardinality of fibers 
holds for $F$ if, and only if, the same bound 
holds for $G$. 
\end{enumerate}
\end{Thm}

\begin{proof}
$N=\infty$ may occur 
 for singular maps, but that does not occasion
any problems. 
In each step below a map $F$ is replaced by a map $G$, 
which becomes the new $F$ for the next step. At each 
step both $F$ and $G$ are nondegenerate, and they 
satisfy both (1) and (2) above, whether singular or not.  
For all but one step, that is true automatically, because 
the step is an equivalence or stable equivalence using 
polynomial automorphisms.

Step 1. Lower the degree. 
Suppose $F=(\comps)$. 
$F$ is polynomially stably equivalent to 
$(f_1 - (y+a)(z+b),f_2,\ldots,f_n,y+a,z+b)$, 
where $a,b$ are polynomials that depend only on 
$\vars$. Thus, if a term of $f_1$ has the form 
$ab$, with $\deg(a)>1$ and $\deg(b)>1$, it 
can be removed at the cost of introducing two new 
variables and some terms of degree less than $\deg(ab)$.
Repeating this for terms of maximum degree until 
there are no more maximal degree 
terms of the specified form in any component, one 
finally obtains a polynomial map $G$ (in a generally 
much higher dimension), all of whose terms are of degree 
no more than three.
This is a standard algorithm \cite{BCW82,ArnoBook}.  
There is flexibility in the choice of term to 
remove next, and one can opportunistically remove a 
product $ab$ that is not a single term, making choices to 
reach a cubic map more quickly. This step is a polynomial 
stable equivalence.

Step 2. Normalize. 
 $F$ is now cubic and (still) nondegenerate. 
Let $n$ be the current dimension. 
Choose $x_0 \in \R^n$ with $j(F)(x_0) \ne 0$. 
After suitable translations, 
$(J(F)(x_0))^{-1}F$ becomes a cubic map $G$, such  
that $G(0)=0$ and $G'(0)=J(G)(0)$ is the identity 
matrix $I$. This step is an affine (in the vector space sense) 
equivalence.

Step 3. Replicate. 
Now $F=X+Q+C$, where $Q$ and $C$ are, respectively, 
 the quadratic and cubic homogeneous components of $F$. 
Let $t$ be a new variable, and put 
$G=(X+tQ+t^2C,t)$. 
This is the step at which nondegeneracy, (1), and (2) will be 
explicitly verified. Let $x$ be any point of $\R^n$. 
For $t \ne 0$, $G(x,t)=(t^{-1}F(tx),t)$, so 
$j(G)(x,t)=j(F)(tx)$, and by continuity of polynomials, 
$j(G)(x,0)=1$. That yields nondegeneracy and  (1). 
For (2), identify $\R^n$ with
 $\R^n \times \{0\}\subset \R^{n+1}$, and
observe that each hyperplane $t=a$ is mapped to itself 
by a map affinely equivalent to $F$ for $a \ne 0$, and 
by the identity for $a=0$.
This step is generally not a stable equivalence.

Step 4. Final step. 
Now $F=(X+tQ+t^2C,t)$, with $Q$ quadratic 
homogeneous and $C$ cubic homogeneous, 
and both independent of $t$. 
Define  two polynomial automorphisms
$A_1,A_2$ in$X,Y,t$, where $Y$ is 
a sequence of $n$ additional variables, 
by $A_1=(X-t^2Y,Y,t)$ and $A_2=(X,Y+C,t)$. 
Then $G=A_1 \circ (X+tQ+t^2C,Y,t) \circ A_2$ is 
the map of cubic homogeneous type 
$(X-t^2Y+tQ,Y+C,t)$. 
This step is a polynomial stable equivalence. 
\end{proof}

Remark. The theorem and proof are valid over $\C$ as well as 
over $\R$, and, indeed, more generally. There are also a
number of preservation results not stated in the theorem. 
For instance, 
$F$  is a Keller map if, and only if, 
$G$ is a Keller map. 
In particular, applying steps 1 through 4 to a Pinchuk map 
yields a Yagzhev map $G$, for which $j(G)$ is not constant 
and $J(G)$ is not unipotent.

On inquiry, both Gianluca Gorni and Michiel de Bondt 
confirmed that Gorni-Zampieri pairing preserves $N$, and 
sent proofs. Since any Yagzhev map can be paired to a 
\Druz map, and nonsingularity is preserved 
in both directions \cite{GZpairs},
there is a further 
reduction of the MSRJC to the cubic linear case.

More recently, reductions of the ordinary JC to the 
symmetric case have been considered, primarily over 
$\R$ and $\C$. 
Let $k$ denote a field of characteristic zero. 
In the JC world  a polynomial map 
$F:k^n \To k^n$ is often called symmetric, in a 
startling abuse of language, if $J(F)$ is a symmetric 
matrix. In that case, $F$ is the gradient map of a polynomial 
function $h:k^n \To k$ and $J(F)$ is the Hessian matrix of second 
order partial derivatives of $h$. So in the symmetric case, 
the JC becomes the Hessian conjecture (HC), namely that gradient 
maps of polynomials with constant nonzero Hessian 
determinant have polynomial inverses. 
In \cite{Meng}, Guowu Meng proves, among many other 
results, the equivalence of the JC and the HC, using what 
he refers to as a trick. Meng's trick is the construction 
featured in the proof below, and works over any $k$. 
In \cite{SymmetricCase}, Michiel 
de Bondt and Arno van den Essen prove a more targeted 
reduction over $\C$, namely to symmetric Keller Yagzhev maps. 
The reduction process involves the use of $\sqrt{-1}$, and 
if applied to a real Keller map may yield a Yagzhev map 
that is not real. Interestingly, they later show 
that all complex symmetric Keller \Druz maps 
have polynomial inverses \cite{SymmetricDmaps}. 

The following theorem 
reduces the entire MRRJC, not just the MSRJC, to 
the symmetric case.

\begin{Thm}\label{reduce-2}
Any nonsingular $\mathcal{C}^2$ map $F:\R^n \To \R^n$,
can be extended to a $\mathcal{C}^1$ 
nonsingular map $G:\R^{2n} \To \R^{2n}$ , 
such that 
\begin{enumerate}
\item the \J matrix of $G$ is symmetric, 
\item $F$ is an everywhere defined real rational map 
if, and only if, $G$ is, 
\item $F$ is polynomial if, and only if, $G$ is, and
\item a finite bound on the cardinality of fibers 
holds for $F$ if, and only if, the same bound 
holds for $G$. 
\end{enumerate}
\end{Thm}

\begin{proof}
Suppose $F=(\comps)$, for twice continuously 
differentiable functions 
$f_i$ in the variables $\vars$. Use coordinates 
$v_1,\ldots,v_n,\vars$ on $\R^{2n}$. Define a real valued 
function on $\R^{2n}$ by 
$h(v,x)=v_1f_1+\cdots+v_nf_n$. 
Let $G$ be the gradient of $h$. 
Then $G(v,x)=(F(x),v\cdot F'(x))$, where $\cdot$ 
denotes a vector matrix product and $F'$ denotes 
$J(F)$, the \J matrix of $F$. 
Viewing $G$ as a map from $\R^{2n}$ to $\R^{2n}$, 
it is $\mathcal{C}^1$ and also clearly satisfies  (2) and (3). 
It satisfies (1) because $J(G)$ is the Hessian matrix of 
a $\mathcal{C}^2$ function. 
To show that $G$ is nonsingular, just note that $J(G)$ has 
a leading $n$ by $n$ block of zeros, flanked 
on the right by $F'$ and below, therefore, by 
the transpose of $F'$, and so 
$j(G)(v,x)=(-1)^n(j(F)(x))^2$. 
Now $G(v,x)=(w,y)$ if, and only if, 
$F(x)=y$ and $v \cdot F'(x)=w$. Since $F'(x)$ is an 
invertible matrix at any point $x$, there is a bijection 
between the inverse images of a point $(w,y) \in \R^{2n}$ under 
$G$ and the inverse images of $y$ under $F$, and that 
clearly establishes (4). 
\end{proof}

Remark. Application to the JC involves noting that 
$F$ is Keller if, and only if, $G$ is, and that $F$ has a 
polynomial inverse if, and only if, $G$ has.

\section{Appendix - supplemental data}\label{sup}

This final section supplies additional data about
the map that is at the heart of this article, namely
the \pk map $F$ of total degree $25$ defined and described in section \ref{specific}. It includes a discussion of how the geometric behavior was determined, equations for the asymptotic variety as a \p curve, and complete details of the minimal \p of section \ref{ext}. 

The key to the geometry is the following table.
The table shows that the number of connected components of a level set $P=c$ can vary from $2$ to $5$.
That number and the range of $Q$ on each connected component can be found by
parametrizing the locus of zeros of
$P-c$ or of  any factors, expressing $q$
as a function of the parameter, and taking limits.
A bit tricky, but important, is the fact that $q^+ <q^-$
for every line of the table in which they appear.
Since $j(F)$ vanishes nowhere, $Q$ is monotone on each connected component of a level set.
The description of the number of inverse images of various points given in section \ref{specific} is then easily verified.

\begin{table}[htb]
\begin{center}
\begin{tabular}{||c||c||}
\hline
$P=c$  & Ranges of Q on the components \\
\hline \hline
$c>0$ & $(-\infty,q+),\; (q+,q-),\; (q-,+\infty),\; (-\infty,+\infty)$ \\
\hline
$c=0$ & $(0,208),\; (-\infty,0),\; (0,+\infty),\; (-\infty,0),\; (208,+\infty)$\\
\hline
$-1<c<0$ & $(-\infty,\; q+),\; (q+,\; +\infty),\; (-\infty,\; q-),\; (q-,\; +\infty)$ \\
\hline$c=-1$ & $(-\infty,-163/4),\; (-\infty,-163/4),\; (-163/4,+\infty),\; (-163/4,+\infty)$ \\
\hline
$c<-1$ & $(-\infty,+\infty),\; (-\infty,+\infty)$ \\
\hline \hline\multicolumn{2}{|| c ||}{Legend: $(a,b)$ denotes the
open interval from $a$ to $b$, with $a < b$;} \\
\multicolumn{2}{|| c ||}
{$q+$ ($q-$) = the value of $Q$ at $h =-1+\sqrt{1+c}$ (resp., $-1-\sqrt{1+c}$);} \\
\hline \hline
\end{tabular}
\end{center}
\caption{Ranges of $Q$ on level sets $P=c$ for Pinchuk's map}
\end{table}

%%%%%%%%%%%%%%%%%%%%%%%%%%%%%
%%%%%%%%%%%%%%%%%%%%%%%%%%%%%%%
Remark. An equivalent table appeared in \cite{PPR}, but a number of incorrect conclusions about $F$ were drawn from it.
The rational parametrization for level sets $P=c$, with $c \ne 0$ and $c \ne -1$, also appeared, unfortunately with a typographical error.
However, the author used  the correct parametrization in deriving the table.
There were corrections in the unpublished manuscript 
\cite{Picturing}
and in \cite{PPRErr}.
The following parametrizations were also used in \cite{PPR}, but not given explicitly there.
For $P=-1$ the parametrizations are 
$x=-t^{-1}-t^{-2},y=-t^2$ for $t \ne 0$ and
$x=-s^2,y=-s^{-2}+s^{-3}-s^{-4}$ for $s \ne 0$.
For $P=0$ they are
$x=-t^{-1},y=-t-t^2$ for $t \ne 0$ (two components) and
$x=-(h+1)h^{-1}(h+2)^{-2},y=-h(h+1)(h+2)^2$ for $h \notin \{0,-2\}$ (three components).

%%%%%%%%%%%%%%%%%%%%%%%%%%%%
%%%%%%%%%%%%%%%%%%%%%%%%%%%%%%%%

The following details about the equations defining $A(F)$ are
reproduced from \cite{aspc}.
$A(F)$ has the bijective polynomial parametrization by  $s \in \R$:
\begin{equation*}
P(s) = s^2 - 1
\end{equation*}
\begin{equation*}
Q(s) = -75s^5 +
\frac{345}{4}s^4 - 29s^3 + \frac{117}{2}s^2
 - \frac{163}{4}
\end{equation*}
and that its points satisfy the minimal \p equation
\begin{equation*}
(Q -(345/4)P^2 -231P -104)^2 = (P+1)^3(75P + 104)^2.
\end{equation*}
These  equations allow the easy computation of the
earlier mentioned points $a$ and $b$ at which the line $P=3$ intersects $A(F)$ (take $s = \pm 2$), and of the point
$(-104/75,-18928/375)$ (approximately $(-1.38,-50.47)$)
in the Zariski closure of $A(F)$ that does not lie on the curve $A(F)$ itself.
Note that for $P=c \ge -1$, $s=h+1$, where $h=-1\pm\sqrt{1+c}$.
This $s$ has nothing to do
with the $s$ used above to parametrize two components of
$P=-1$.

In section \ref{ext}  a \p $R(T)$ was defined, but not fully written out.
It has degree $6$ in $T$, coefficients in $\Q[P,Q]$, and $h$
as a root in $\R[x,y]$.
It was shown in Proposition \ref{R=cm} that $R(T)=(197/4)m(T)$,
where $m(T)$ is the monic minimal \p of $h$ over both $\R[P,Q]$ and $\R(P,Q)$.
Straightforward computations show that
\begin{align}\label{fullR}
R(T) &=(197/4)T^6 + (104 -(363/2)P)T^5
+(63-421P+(825/4)P^2)T^4\tag{6}\\
 &+(-306P+510P^2-75P^3)T^3
+(-Q+412P^2-195P^3)T^2\notag\\
 &+(2PQ-170P^3)T -P^2Q,\notag
\end{align}
and  this formula makes it trivial to evaluate the effect of setting
$P$ and/or $Q$ equal to zero.

The following  partial derivatives of $R(T)$ were used, 
but not fully written out, in the example in section \ref{conj4B}.

\begin{align*}
\partial R/\partial P &=
(-363/2)T^5+(-421+(825/2)P)T^4 \\
&+(-306+1020P-225P^2)T^3+(824P-585P^2)T^2 \\
&+(2Q-510P^2)T -2PQ
\end{align*}

\begin{align*}
\partial R/\partial Q &=
 -T^2 +(2P)T -P^2 = -(T-P)^2
\end{align*}

If one sets $P=T$ in the expression for $\partial R/\partial P$ above, then the terms of degrees $0$ and $1$ in $T$ drop out, 
and the result is $T^3$ times the following quadratic polynomial
in $T$ alone. 

\begin{align*}
 &T^2(-363/2+825/2-225)\\
&+T(-421+1020-585)\\
&+(-306+824-510)\\
&= 6T^2 + 14T + 8
\end{align*}

The same example postponed to this appendix the 
verification of the location of the 
three component curves 
of $B(F)$ relative to the branches of $S$.

The image of the $y$-axis is easily shown to be the straight line 
$4Q=200P+33$. The line contains the points $(0,33/4)$ and 
$(-1,-167/4)$ of the $(P,Q)$-plane, which 
lie, respectively, above and below the points $(0,0)$ and 
$(-1,-163/4)$ of $A(F)$. 
So the line crosses $A(F)$ between those points. 
 Since the crossing point has only one inverse 
image, the $y$-axis intersects the component curve of   
$B(F)$ on which $P$ is bounded above. 
Since $P=y$ on the $y$-axis, the intersection point 
lies in the 'top' region.

The rational parametrization of the level set $P = 3$, 
described in some detail near the beginning of section \ref{ext}, 
is given,   in part, by the curve $(x(h),y(h))$ for $h>3$. 
The image of that  
curve is the entire line $P = 3$, which crosses $A(F)$ twice. 
So the curve itself crosses both the other 
components of $B(F)$, by the same reasoning as before.  
At $h=4$, the parametrization yields the point
$A=(-5/441, 441(-17))$ of the $(x,y)$-plane. 
The vertical line $x=-5/441$ intersects three of the 
four branches of $S$, each at a single point, in the order, 
from top to bottom, of $x=1/y+1/\sqrt{-y}$ ($y<0$) first, 
then $y=1/x-1/x^2$ ($x<0$), and finally
$x=1/y-1/\sqrt{-y}$ ($y<0$). 
The point $A$ is between the first two points of intersection, 
and the two component curves of $B(F)$ at issue must lie 
in the same region of $S^c$ as $A$. 
Interestingly, the just added curve and the $y$-axis do not 
intersect, even though their images obviously do.

%\bibliographystyle{plain}
%\bibliography{localbib,samuelson,reduction}                             

\end{document}